\documentclass[11pt]{article}
\usepackage{amsthm, amsmath, amssymb, amsfonts, url, booktabs, tikz, setspace, fancyhdr, bm}
\usepackage{hyperref}
\usepackage{geometry}
\usepackage{hyperref, enumerate}
\usepackage[shortlabels]{enumitem}
\usepackage[babel]{microtype}
\usepackage[english]{babel}
\usepackage[capitalise]{cleveref}
\usepackage{comment}
\usepackage{bbm}
\usepackage{csquotes}
\usepackage{mathabx}
\usepackage{tikz}
\usepackage{graphicx}
\usepackage{float}

\counterwithin{figure}{section}

\newtheorem{theorem}{Theorem}[section]
\newtheorem{prop}[theorem]{Proposition}
\newtheorem{lemma}[theorem]{Lemma}
\newtheorem{cor}[theorem]{Corollary}
\newtheorem{conj}{Conjecture}
\newtheorem{claim}[theorem]{Claim}

\theoremstyle{definition}
\newtheorem{defn}[theorem]{Definition}
\newtheorem*{defn-non}{Definition}

\newtheorem{rmk}[theorem]{Remark}

\newenvironment{poc}{\begin{proof}[Proof of claim]}{\end{proof}}

\usepackage{todonotes}

\newcommand{\ex}{\mathrm{ex}}
\newcommand{\cP}{\mathcal{P}}
\newcommand{\cQ}{\mathcal{Q}}

\title{Extremal number of graphs from geometric shapes}
\author{
Jun Gao\thanks{Extremal Combinatorics and Probability Group (ECOPRO), Institute for Basic Science (IBS), Daejeon, South Korea. Emails: {\texttt \{jungao, hongliu, zixiangxu\}@ibs.re.kr}. Supported by IBS-R029-C4.}
\and 
Oliver Janzer\thanks{Department of Pure Mathematics and Mathematical Statistics, University of Cambridge, United Kingdom. Email: oj224@cam.ac.uk.}
\and
Hong Liu\footnotemark[1]
\and
Zixiang Xu\footnotemark[1]
}

\begin{document}
\maketitle

\begin{abstract}
    We study the Tur\'{a}n problem for highly symmetric bipartite graphs arising from geometric shapes and periodic tilings commonly found in nature.
    \begin{enumerate}
        \item The prism $C_{2\ell}^{\square}:=C_{2\ell}\square K_{2}$ is the graph consisting of two vertex disjoint $2\ell$-cycles and a matching pairing the corresponding vertices of these two cycles. We show that for every $\ell\ge 4$, ex$(n,C_{2\ell}^{\square})=\Theta(n^{3/2})$. This resolves a conjecture of He, Li and Feng.
        \item The hexagonal tiling in honeycomb is one of the most natural structures in the real world. We show that the extremal number of honeycomb graphs has the same order of magnitude as their basic building unit 6-cycles.
        \item We also consider bipartite graphs from quadrangulations of the cylinder and the torus. We prove near optimal bounds for both configurations. In particular, our method gives a very short proof of a tight upper bound for the extremal number of the 2-dimensional grid, improving a recent result of Brada\v{c}, Janzer, Sudakov and Tomon.
    \end{enumerate}
    Our proofs mix several ideas, including shifting embedding schemes, weighted homomorphism and subgraph counts and asymmetric dependent random choice.
\end{abstract}
\begin{figure}[htbp]
\centering
\begin{minipage}[t]{0.32\textwidth}
\centering
\includegraphics[width=5cm]{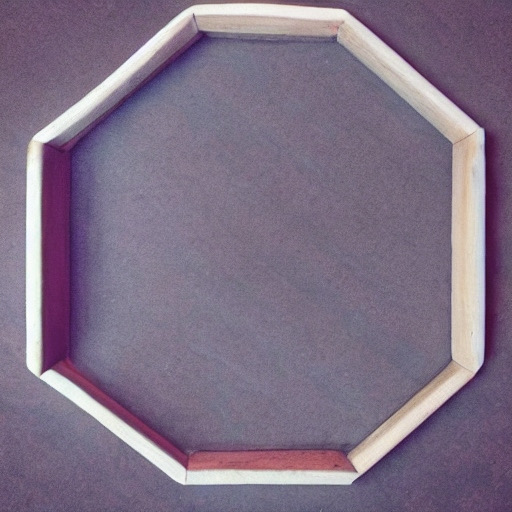}
\end{minipage}
\begin{minipage}[t]{0.32\textwidth}
\centering
\includegraphics[width=5cm]{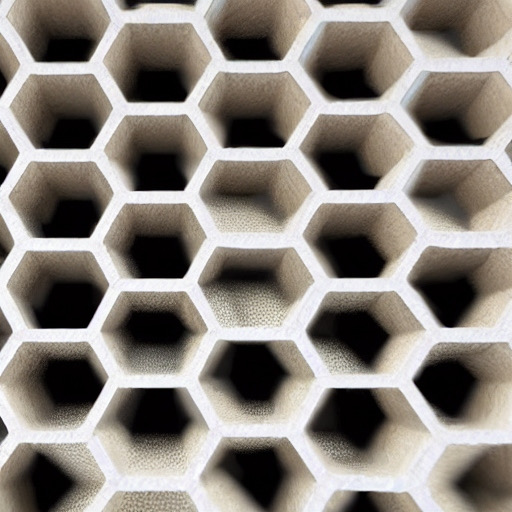}
\end{minipage}
\begin{minipage}[t]{0.32\textwidth}
\centering
\includegraphics[width=5.45cm]{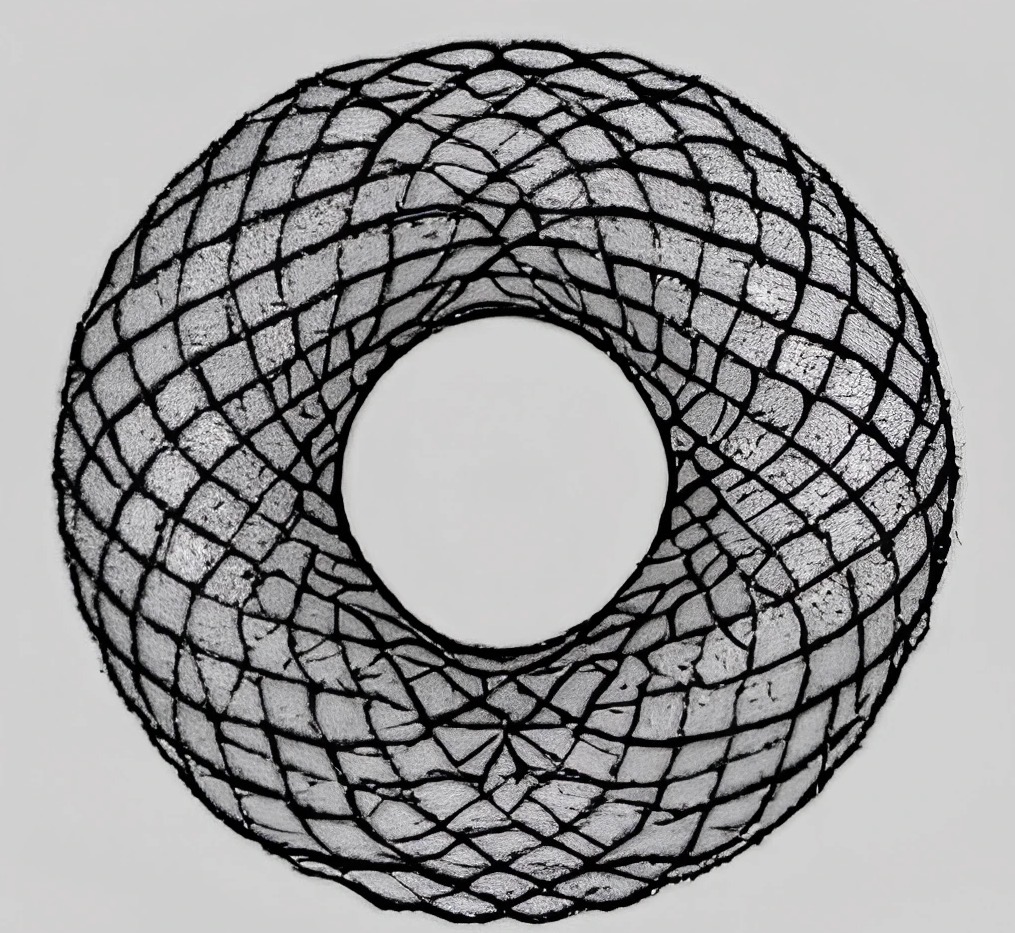}
\end{minipage}
\end{figure}

\section{Introduction}
Understanding symmetric structures is a prominent objective across various branches of mathematics due to their exceptional regularity and aesthetic appeal. Studies on such structures not only serve as a pursuit of beauty but also form the foundation for several fields of mathematics. In this paper, we investigate the Tur\'{a}n problem for highly symmetric bipartite graphs corresponding to various geometric structures.

The Tur\'{a}n problem, one of the most central topics in extremal combinatorics, is concerned with determining the maximum possible density a graph can have without containing a given graph as a subgraph. Its origin dates back to the result of Mantel~\cite{1907Mantel} in $1907$, stating that every $n$-vertex triangle-free graph has at most $\frac{n^{2}}{4}$ edges, and the work of Erd\H{o}s~\cite{1944Erdos} that connects the study of Sidon sets with the Tur\'an problem for 4-cycles. Tur\'{a}n~\cite{1941Turan} in 1941 extended Mantel's theorem to graphs without a copy of the clique $K_r$. In the same paper, Tur\'an proposed the study of the five graphs corresponding to platonic solids, and his result covers the tetrahedron graph $K_4$. The problem of octahedron, dodecahedron and icosahedron graphs were later resolved by Erd\H{o}s and Simonovits~\cite{1971ErdosSimonovits} and by Simonovits~\cite{1974DMSimonovits,1974JCTBSimonovits} respectively, while the innocent looking cube graph remains elusive. Apart from the aforementioned study of Sidon sets in combinatorial number theory, the Tur\'an problem has connections to other fields and during its course of advancement, many tools and methods have been developed, see e.g. the use of (random) algebraic and finite geometric structures~\cite{1999ProjectiveNorm,Brown1966,2015BLMSBukh,2018JEMSBukh,Erdos1966,1996Norm,1995BAMS,WengerC4C6C10} for lower bound constructions, and applications of Tur\'an type results in discrete geometry~\cite{1992CPCHull,1946AMMErdos} and information theory~\cite{2022TIT,2018TITCaching,1997TIT}. 

Formally, for a graph $F$, the \emph{extremal number of $F$}, denoted by $\ex(n,F)$, is the maximum number of edges in an $n$-vertex graph not containing $F$ as a subgraph. The Erd\H{o}s-Stone-Simonovits Theorem~\cite{1966ESS, 1946ErdosBAMS}, considered to be a fundamental theorem in extremal graph theory, states that
$$\ex(n,F)=\bigg(1-\frac{1}{\chi(F)-1}+o(1)\bigg)\binom{n}{2},$$
where $\chi(F)$ is the chromatic number of $F$. This theorem asymptotically solves the problem when $\chi(F)\ge 3$. However, for bipartite graphs, not even the order of magnitude is known in general. It also explains why the cube graph is the most difficult among the five platonic solids as it is the only one which is bipartite.

Two fundamental classes of bipartite graphs with high symmetry are even cycles and complete bipartite graphs. For the even cycle $C_{2\ell}$, Bondy and Simonovits~\cite{1974BondyEvenCycle} gave an upper bound $\ex(n,C_{2\ell})=O_{\ell}(n^{1+1/\ell})$, see~\cite{2021EJCEvenCycle} for the current best bound. Matching lower bound constructions are only known for the $4$-, $6$-, and $10$-cycle, utilizing generalized polygons from finite geometry, see~\cite{Brown1966,Erdos1966,WengerC4C6C10}. The order of magnitude of $\ex(n,C_{2\ell})$ is still unknown for any $\ell\notin\{2,3,5\}$. For the complete bipartite graphs $K_{s,t}$, a well-known result of K\H ov\'ari, S\'{o}s and Tur\'{a}n~\cite{1954KST} in 1954 showed that $\textup{ex}(n,K_{s,t})=O(n^{2-1/s})$ for any integers $t\geqslant s$. Erd\H{o}s, R\'{e}nyi and S\'{o}s~\cite{Erdos1966} and Brown~\cite{Brown1966} respectively proved matching lower bounds for the cases $s=2$ and $s=3$. For larger $s$, Koll\'{a}r, R\'{o}nyai and Szab\'{o}~\cite{1996Norm} first showed that $\text{ex}(n,K_{s,t})=\Omega(n^{2-1/s})$ when $t\geqslant s!+1$, and then the bound on $t$ was improved to $t\geqslant (s-1)!+1$ by Alon, R\'{o}nyai and Szab\'{o}~\cite{1999ProjectiveNorm}. The current best dependence $t\ge 9^{s}\cdot s^{4s^{2/3}}$ was shown by Bukh~\cite{2021arxivBukh}. For more on the bipartite Tur\'an problem, we refer the readers to the comprehensive survey of F\"{u}redi and Simonovits~\cite{2013Survey}.

In this paper, we continue this line of study and determine the order of magnitude of the extremal number for highly symmetric bipartite graphs stemming from certain geometric shapes and periodic tilings.

\subsection{The prisms}
The \emph{$2\ell$-prism} $C_{2\ell}^{\square}:=C_{2\ell}\square K_{2}$ is the Cartesian product of the $2\ell$-cycle and an edge. In other words, $C_{2\ell}^{\square}$ consists of two vertex disjoint $2\ell$-cycles and a matching joining the corresponding vertices on these two cycles. As $C_{2\ell}^{\square}$ contains many $4$-cycles, we have a lower bound $\ex(n,C_{2\ell}^{\square})\ge \ex(n,C_4)=\Omega(n^{3/2})$. Note that $C_{4}^{\square}$ is the notorious cube graph, and the best known upper bound is $\ex(n,C_{4}^{\square})= O(n^{8/5})$~\cite{1970Erdos, 2005Sharir}. Studying $C_{2\ell}^{\square}$ could shed some light on the cube problem. An upper bound $\ex(n,C_{2\ell}^{\square})=O(n^{5/3})$ can be easily obtained via the celebrated dependent random choice method~\cite{AKS03,2011DRCSurvey}.

Very recently, He, Li and Feng~\cite{2023arxivOddCylinder} studied the odd prisms, determined $\ex(n,C_{2k+1}^{\square})$ for any $k\ge 1$ for large $n$ and characterized
the extremal graphs. They proposed the following conjecture to break the $5/3$ barrier for $2\ell$-prism.
\begin{conj}[\cite{2023arxivOddCylinder}]\label{conj:NewConj2023}
    For every $\ell\ge 2$, there exists $c=c(\ell)>0$ such that $\ex(n,C_{2\ell}^{\square})=O(n^{5/3-c})$.
\end{conj}

Our first result provides an optimal upper bound for $C_{2\ell}^{\square}$ for every $\ell\ge 4$.

\begin{theorem}\label{thm:Main1}
  For any integer $\ell\geqslant 4$, we have 
  \begin{equation*}
    \ex(n,C_{2\ell}^{\square})=\Theta_{\ell}(n^{3/2}).
  \end{equation*}
\end{theorem}

We remark that larger prisms are easier to handle. In Subsection~\ref{shorter proof}, we provide a shorter and different proof of $\ex(n,C_{2\ell}^{\square})=O_{\ell}(n^{3/2})$ for $\ell\ge 7$, which also yields $\ex(n,C_{6}^{\square})=O(n^{21/13}(\log{n})^{24/13})$. This, together with the known bound for the cube and Theorem~\ref{thm:Main1}, confirms Conjecture~\ref{conj:NewConj2023}.

It is worth mentioning that Erd\H{o}s offered \$250 for a proof and \$500 for a counterexample of the following conjecture. A graph is \emph{$r$-degenerate} if each of its subgraphs has minimum degree at most $r$.
\begin{conj}[\cite{1981ESconj}]\label{conj:r=2}
    Let $H$ be a bipartite graph. Then $\ex(n,H)=O(n^{3/2})$ if and only if $H$ is $2$-degenerate.
\end{conj}

This conjecture was recently disproved by Janzer~\cite{2021arxivJanzer}, who constructed, for each $\varepsilon>0$, a $3$-regular bipartite graph $H$ with girth 6 such that $\ex(n,H)=O(n^{4/3+\varepsilon})$. Theorem~\ref{thm:Main1} provides a new family of 3-regular \emph{girth-4} counterexamples to Conjecture~\ref{conj:r=2}.

\subsection{The honeycomb}
The hexagonal tiling in honeycomb is one of the most common geometric structures, appearing in nature in many crystals. It is also the densest way to pack circles in the plane.
As the honeycomb graph $H$ of any size contains $C_6$ as a subgraph, we have a lower bound $\ex(n,H)\ge \ex(n,C_6)=\Omega(n^{4/3})$.

Our second result is a matching upper bound $O(n^{4/3})$, showing that the hexagonal tiling appears soon after the appearance of a single hexagon. In particular, we consider the following graph $H_{k,\ell}$ (see Figure~\ref{figure:honeycomb}), which contains any (finite truncation of a) honeycomb graph as a subgraph when $k$ and $\ell$ are sufficiently large.

\begin{defn-non}
For an odd integer $k\ge 1$ and even integer $\ell \ge 2$, let $H_{k,\ell}$ be the graph with vertex set $V(H_{k,\ell})=\{ x_{i,j}: 1\le i\le k, 1\le j\le \ell \}$, where $x_{k,1}=x_{k,3} = \cdots = x_{k,\ell-1}=u$ and
$x_{1,2}=x_{1,4} = \cdots = x_{1,\ell}=v$ (but all the other vertices are distinct) and edge set
\begin{align*}
  E(H_{k,\ell})=\{x_{i,j}x_{i,j+1}: 1\le i \le k, 1\le j\le \ell-1\}&\cup \{x_{2i-1,j}x_{2i,j}: 1\le i\le k/2, 1\le j \le \ell,\  j \text  { is odd} \}\\  
  &\cup\{x_{2i,j}x_{2i+1,j}: 1\le i\le k/2, 1\le j \le \ell,\  j \text  { is even} \}.
\end{align*}
\end{defn-non}
In Figure \ref{figure:honeycomb}, the edges $x_{i,j}x_{i,j+1}$ are coloured blue, the edges $x_{2i-1,j}x_{2i,j}$ (with $j$ odd) are coloured red, and the edges $x_{2i,j}x_{2i+1,j}$ (with $j$ even) are coloured green.

\begin{figure}[h]\label{figure:honeycomb}
\centering

\begin{minipage}[t]{0.48\textwidth}
    \centering
    \begin{tikzpicture}[scale=0.45]
    \foreach \n in {0,1,2,3,4,5}
    \foreach \m in {0,1,2,3,4,5}{
    \node[shape = circle,draw = black,fill,inner sep=0pt,minimum size=1.0mm] at (3*\n+1,2*0.866*\m) {};
    \node[shape = circle,draw = black,fill,inner sep=0pt,minimum size=1.0mm] at (3*\n+0.5,2*0.866*\m-0.866) {};
    \node[shape = circle,draw = black,fill,inner sep=0pt,minimum size=1.0mm] at (3*\n-0.5,2*0.866*\m-0.866) {};
    \node[shape = circle,draw = black,fill,inner sep=0pt,minimum size=1.0mm] at (3*\n-1,2*0.866*\m) {};
    \node[shape = circle,draw = black,fill,inner sep=0pt,minimum size=1.0mm] at (3*\n-0.5,2*0.866*\m+0.866) {};
    \node[shape = circle,draw = black,fill,inner sep=0pt,minimum size=1.0mm] at (3*\n+0.5,2*0.866*\m+0.866) {};
    \draw[thick, fill=black, fill opacity=0.3] 
        (3*\n+1,2*0.866*\m) -- (3*\n+0.5,2*0.866*\m-0.866);
    \draw[thick, fill=black, fill opacity=0.3] 
        (3*\n+0.5,2*0.866*\m-0.866) -- (3*\n-0.5,2*0.866*\m-0.866);
    \draw[thick, fill=black, fill opacity=0.3] 
        (3*\n-0.5,2*0.866*\m-0.866) -- (3*\n-1,2*0.866*\m);
    \draw[thick, fill=black, fill opacity=0.3] 
        (3*\n-1,2*0.866*\m) -- (3*\n-0.5,2*0.866*\m+0.866);
    \draw[thick, fill=black, fill opacity=0.3] 
        (3*\n-0.5,2*0.866*\m+0.866) -- (3*\n+0.5,2*0.866*\m+0.866);
    \draw[thick, fill=black, fill opacity=0.3] 
        (3*\n+0.5,2*0.866*\m+0.866) -- (3*\n+1,2*0.866*\m);
    }
    \foreach \n in {0,1,2,3,4}
    \foreach \m in {0,1,2,3,4,5}{
    \draw[thick, fill=black, fill opacity=0.3] 
    (3*\n+1,2*0.866*\m)--(3*\n+2,2*0.866*\m);
    }    
    \end{tikzpicture}
    \caption{Honeycomb}
    \label{fig:Fengwo}
\end{minipage}
\begin{minipage}[t]{0.48\textwidth}
\centering
\begin{tikzpicture}[scale=0.4]
\foreach \m in {0,1,2,3,4,5}
\foreach \n in {0,0.5,1}{
    \node[shape = circle,draw = blue,fill,inner sep=0pt,minimum size=1.0mm] at (8*\n,2*\m) {};

    \node[shape = circle,draw = black,fill,inner sep=0pt,minimum size=1.0mm] at (8*\n+1,2*\m) {};
    \draw[thick, red, fill opacity=0.3] (8*\n,2*\m) -- (8*\n+1,2*\m);
    
    \node[shape = circle,draw = black,fill,inner sep=0pt,minimum size=1.0mm] at (8*\n+2,2*\m+1) {};

    \node[shape = circle,draw = black,fill,inner sep=0pt,minimum size=1.0mm] at (8*\n+3,2*\m+1) {};

    \draw[thick, green, fill opacity=0.3] (8*\n+2,2*\m+1) -- (8*\n+3,2*\m+1);

    \draw[thick, blue, fill opacity=0.3] (8*\n+1,2*\m) -- (8*\n+2,2*\m+1);
}   

\foreach \m in {0,1,2,3,4}
\foreach \n in {0,0.5,1}{
    \draw[thick, blue, fill opacity=0.3] 
    (8*\n+1,2*\m+2) -- (8*\n+2,2*\m+1);
}

\foreach \m in {0,1,2,3,4,5}
\foreach \n in {0.5,1}{
    \draw[thick, blue, fill opacity=0.3] 
    (8*\n,2*\m) -- (8*\n-1,2*\m+1);
}

\foreach \m in {0,1,2,3,4}
\foreach \n in {0.5,1}{
    \draw[thick, blue, fill opacity=0.3] (8*\n,2*\m+2) -- (8*\n-1,2*\m+1);
}
\node[shape = circle,draw = black,fill,inner sep=0pt,minimum size=1.0mm] at (-2,6) {};
\node[left] at (-2,6) {$v$};
\node[shape = circle,draw = black,fill,inner sep=0pt,minimum size=1.0mm] at (13,7) {};
\node[right] at (13,7) {$u$};

\node[below] at (-0.5,0) {$x_{1,1}$};
\node[below] at (1.5,0) {$x_{2,1}$};
\node[below] at (3.5,0) {$x_{3,1}$};
\node[below] at (5.5,0) {$x_{4,1}$};
\node[below] at (7.5,0) {$x_{5,1}$};
\node[below] at (9.5,0) {$x_{6,1}$};

\node[above] at (1.5,11) {$x_{2,12}$};
\node[above] at (3.5,11) {$x_{3,12}$};
\node[above] at (5.5,11) {$x_{4,12}$};
\node[above] at (7.5,11) {$x_{5,12}$};
\node[above] at (9.5,11) {$x_{6,12}$};
\node[above] at (11.5,11) {$x_{7,12}$};

\foreach \m in {0,1,2,3,4,5}{
\draw[thick, blue, fill opacity=0.3] (-2,6) -- ( 0,2*\m);
\draw[thick, blue, fill opacity=0.3] (13,7) --(11,2*\m+1);
}
\end{tikzpicture}

\caption{$H_{7,12}$}
\end{minipage}

\end{figure}

\begin{theorem}\label{thm:tightHoneyComb}
For any odd integer $k\ge 1$ and even integer $\ell \ge 2$,
$$\ex (n, H_{k,\ell}) = \Theta_{k,\ell}(n^{4/3}).$$
\end{theorem}

\subsection{The grid}

We will also give an improved bound for the extremal number of the grid. For a positive integer $t$, $F_{t,t}$ is the graph with vertex set $[t]\times [t]$ in which two vertices are joined by an edge if they differ in exactly one coordinate and in that coordinate they differ by one. Brada\v{c}, Janzer, Sudakov and Tomon~\cite{2022TuranGrid} determined the extremal number of $F_{t,t}$ up to a multiplicative constant which depends on $t$, showing that for any $t\ge 2$,
$$\Omega(t^{1/2}n^{3/2})\le \ex(n,F_{t,t})\le e^{O(t^{5})}n^{3/2}.$$

They have asked to determine the correct dependence on $t$. We make substantial progress on this question by giving a very short proof of the following bound, which shows that the dependence on $t$ is polynomial.

\begin{theorem} \label{thm:normal grid}
    For any positive integer $t$, if $n$ is sufficiently large in terms of $t$, then $$\ex(n,F_{t,t})\leq 5t^{3/2}n^{3/2}.$$
\end{theorem}

It would be interesting to determine the correct power of $t$ in $\ex(n,F_{t,t})$.

\subsection{Quadrangulations of cylinder and torus}
Next, we consider certain quadrangulations of the cylinder and the torus, see Figure \ref{fig:torus}.
\begin{defn-non}[Quadrangulation of a cylinder]
    For integers $k,\ell\ge 2$, let $P_{k,\ell}$ be the graph with vertex set $V(P_{k,\ell}) = \{ x_{i,j}: 1\le i \le k,1 \le j\le \ell  \}$, and edge set
\begin{align*}
  E(P_{k,\ell})=\{x_{i,j}x_{i+1,j}: 1\le i \le k-1, 1\le j\le \ell\}&\cup \{x_{i,j+1}x_{i+1,j}: 1\le i\le k-1, 1\le j \le \ell, \  i \text  { is odd}\}\\  
  &\cup\{x_{i,j}x_{i+1,j+1}: 1\le i\leq k-1, 1\le j \le \ell, \  i \text  { is even}\},
\end{align*}
where $x_{i,\ell+1}=x_{i,1}$ for all $i\in [k]$.
\end{defn-non}

Clearly, the extremal number of such a quadrangulated cylinder is at least that of the 4-cycle. Our next result infers that in fact they are of the same order of magnitude.
\begin{figure}[htbp]\label{figure:P55}
\centering
\begin{tikzpicture}[scale=0.8]

\begin{scope}[xshift=3cm,yshift=-3cm]
    \foreach \m in {1,2,3,4,5}
\foreach \n in {2,3,4}{

\ifthenelse {\n =1}{
\node[shape = circle,draw = blue,fill=blue,inner sep=0pt,minimum size=1.5mm] at (\m,\n) {};
\node[below] at (\m,\n) {$x_{\m,\n}$};
}{
\ifthenelse {\n =5}{
\node[shape = circle,draw = blue,fill=blue,inner sep=0pt,minimum size=1.5mm] at (\m,\n) {};
\node[above] at (\m,\n) {$x_{\m,4}$};
}{
{    \node[shape = circle,draw = black,fill=black,inner sep=0pt,minimum size=1.5mm] at (\m,\n) {};}
}}}

\foreach \m in {2,4,6,8,10}
\foreach \n in {1,5}{
\node[shape = circle,draw = blue,fill=blue,inner sep=0pt,minimum size=1.5mm] at ({\m/2},\n) {};
\ifthenelse {\n =5}{
\node[above] at ({\m/2},\n) {$x_{\m,4}$};
}
{\node[below] at ({\m/2},\n) {$x_{\m,4}$};}
}

\foreach \n in {1,2,3,4}{
\node[left] at ({1/2},{5-\n+1/2}) {$x_{1,\n}$};
}

\foreach \m in {1,2,3,4,5,6}
\foreach \n in {1,2,3,4}{
\ifthenelse {\m =1}{
\node[shape = circle,draw = black,fill=red,inner sep=0pt,minimum size=1.5mm] at ({\m-1/2},{\n+1/2}) {};}
{\ifthenelse {\m =6}{
\node[shape = circle,draw = black,fill=red,inner sep=0pt,minimum size=1.5mm] at ({\m-1/2},{\n+1/2}) {};}
{\node[shape = circle,draw = black,fill=black,inner sep=0pt,minimum size=1.5mm] at ({\m-1/2},{\n+1/2}) {};}}

\ifthenelse {\m <6}{
\draw[black] (\m,{\n+1})--({\m-1/2},{\n+1/2})--(\m,\n);
\draw[black] (\m,{\n+1})--({\m+1/2},{\n+1/2})--(\m,\n);}
}

\end{scope}

\begin{scope}[xshift=18cm]

\foreach \m in {1,2,3,4,5,6}
\foreach \n in {0,120,240}{
\node[shape = circle,draw = black,fill=black,inner sep=0pt,minimum size=1.5mm] at ({\m*sin(\n+\m*60)/2 },{\m*cos(\n+\m*60)/2}) {};
}

\node[above] at ({3*sin(0+6*60) },{3*cos(0+6*60)}) { $x_{1,1}$};
\node[left] at ({2.5*sin(0+5*60) },{2.5*cos(0+5*60)}) { $x_{2,3}$};
\node[right] at ({3*sin(120+6*60) },{3*cos(120+6*60)}) { $x_{1,2}$};
\node[right] at ({2.5*sin(120+5*60) },{2.5*cos(120+5*60)}) { $x_{2,1}$};
\node[left] at ({3*sin(240+6*60) },{3*cos(240+6*60)}) { $x_{1,3}$};
\node[below] at ({2.5*sin(240+5*60) },{2.5*cos(240+5*60)}) { $x_{2,2}$};

\foreach \n in {0,120,240}{
\draw[black] ({2.5*sin(\n-60)},{2.5*cos(\n-60)})--({3*sin(\n)},{3*cos(\n)})--({2.5*sin(\n+60)},{2.5*cos(\n+60)});
}
\foreach \n in {0,120,240}{
\draw[blue] ({2*sin(\n)},{2*cos(\n)})--({2.5*sin(\n+60)},{2.5*cos(\n+60)})--({2*sin(\n+120)},{2*cos(\n+120)});
}
\foreach \n in {0,120,240}{
\draw[red] ({1.5*sin(\n-60)},{1.5*cos(\n-60)})--({2*sin(\n)},{2*cos(\n)})--({1.5*sin(\n+60)},{1.5*cos(\n+60)});
}
\foreach \n in {0,120,240}{
\draw[green] ({1*sin(\n)},{1*cos(\n)})--({1.5*sin(\n+60)},{1.5*cos(\n+60)})--({1*sin(\n+120)},{1*cos(\n+120)});
}
\foreach \n in {0,120,240}{
\draw[purple] ({0.5*sin(\n-60)},{0.5*cos(\n-60)})--({1*sin(\n)},{1*cos(\n)})--({0.5*sin(\n+60)},{0.5*cos(\n+60)});
}

\end{scope}
\end{tikzpicture}
\caption{In the first graph, identifying the blue vertices (in the same column) yields a copy of $P_{11,4}$; if additionally the red vertices (in the same row) are identified, then we obtain a copy of $T_{10,4}$. The second graph is a different way of drawing $P_{6,3}$.}
\label{fig:torus}
\end{figure}
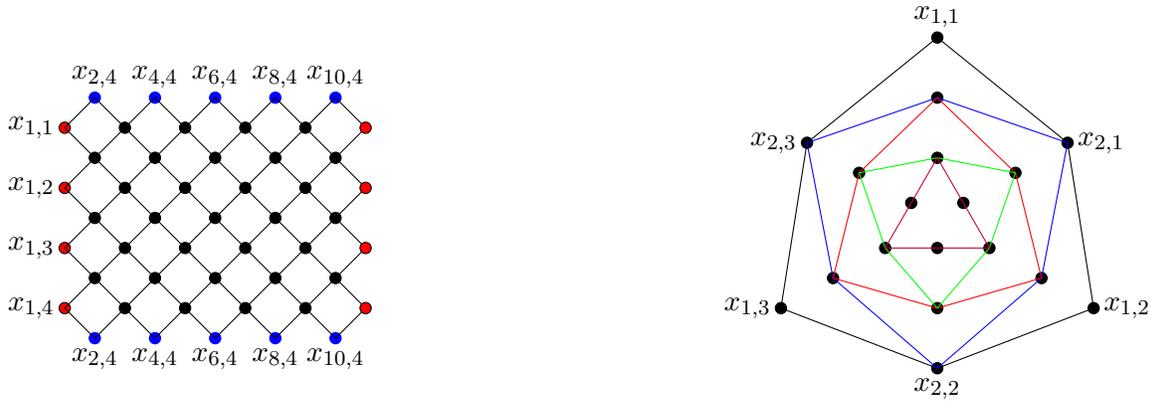

\begin{theorem}\label{thm:grid}
Let $k$ and $\ell$ be positive integers. Then we have
$$\ex(n,P_{k,\ell}) = \Theta_{k,\ell}(n^{3/2}).$$
\end{theorem}

If $k$ is even and we glue the two sides of the cylinder $P_{k+1,\ell}$, then we obtain a torus.
\begin{defn-non}[Quadrangulation of a torus]
    For an even integer $k\geq 4$ and integer $\ell\geq 2$, let $T_{k,\ell}$ be the graph with vertex set $V(T_{k,\ell}) = \{ x_{i,j}: 1\le i \le k,1 \le j\le \ell  \}$, and edge set
\begin{align*}
  E(T_{k,\ell})=\{x_{i,j}x_{i+1,j}: 1\le i \le k, 1\le j\le \ell\}&\cup \{x_{i,j+1}x_{i+1,j}: 1\le i\le k, 1\le j \le \ell, \  i \text  { is odd}\}\\  
  &\cup\{x_{i,j}x_{i+1,j+1}: 1\le i\leq k, 1\le j \le \ell, \  i \text  { is even}\},
\end{align*}
    where $x_{k+1,j} = x_{1,j}$ for all $j\in [\ell]$ and $x_{i,1} = x_{i,\ell +1}$ for all $i\in [k]$.
\end{defn-non}

For the quadrangulated torus, we provide a general upper bound as follows.
\begin{theorem}\label{thm:torus}
For an even integer $k\geq 4$ and an integer $\ell \ge 2$, we have
\begin{equation*}
    \ex(n,T_{k,\ell}) = O_{k,\ell}(n^{\frac{3}{2}+\frac{\ell}{k}}(\log n)^{2}).
\end{equation*}
\end{theorem}

Thus, when $k$ is sufficiently large compared to $\ell$, the exponent can be arbitrarily close to $3/2$. On the other hand, the exponent is always strictly greater than $3/2$ as the probabilistic deletion method (see, e.g, Theorem 2.26 in \cite{2013Survey}) yields the lower bound $\ex(n,T_{k,\ell})=\Omega_{k,\ell}(n^{\frac{3}{2}+\frac{3}{4k\ell-2}})$.

\medskip
\noindent\textbf{Structure of the paper.}
The rest of this paper is organized as follows. We list some useful lemmas in Section~\ref{sec:Preliminary}. The proofs of the main results are given in Sections~\ref{sec:Grids},~\ref{sec:honeycomb} and~\ref{sec:prisms} respectively. Finally we discuss some related problems in Section~\ref{sec:conclude}.

\section{Preliminaries}\label{sec:Preliminary}
{\bf \noindent Notations.}
For a graph $G$, a subset of vertices $T\subseteq V(G)$, let $G-T$ be the graph obtained from $G$ by deleting the vertices in $T$ and all edges incident to any vertices in $T$. We write $G-v$ instead of $G-\{v\}$. For an edge $e\in E(G)$, let $G-e$ be the spanning subgraph obtained from $G$ by deleting the edge $e$. We write $v(G)$ and $e(G)$ for the number of vertices and edges of $G$, respectively. For a subset $T\subseteq V(G)$, we use $N_{G}(T)$ (or $N(T)$ if the subscript is clear) to denote the set of common neighbors of $T$, that is, $N_{G}(T):=\{v\in V(G):uv\in E(G)\ \text{for every vertex\ } u\in T\}$. We use $d_{G}(x)$ (or $d(x)$) to denote the degree of $x$ and $d_{G}(x,y)$ or ($d(x,y)$) to denote the codegree of $x$ and $y$. We write $\Delta(G)$ and $\delta(G)$ for the maximum and minimum degree of $G$, respectively. The notation $P_{k}$ refers to the path with $k$ vertices, i.e. the path of length $k-1$. Moreover, for a path of the form $x_{1}x_{2}\cdots x_{k}$, the \emph{endpoints} are $x_{1}$ and $x_{k}$. We will use $x_{1}x_{2}\cdots x_{2\ell}x_{1}$ to denote a $2\ell$-cycle. Sometimes we also use $(a,b,c,d)$ to denote the $4$-cycle $abcda$. The notation $O(\cdot)$ and $\Omega(\cdot)$ have their usual asymptotic meanings. We use $\log$ throughout to denote the base $2$ logarithm.

\smallskip

The following lemma is folklore.

\begin{lemma}\label{lem:SubgraphLargeDeg}
   If $G$ is a graph with average degree $d$ then it contains a subgraph $G_{1}$ with $e(G_{1})\ge\frac{e(G)}{2}$ and $\delta(G_{1})\ge\frac{d}{4}$.
\end{lemma}

We say that a graph $F$ is \emph{$K$-almost regular} if $\Delta(F)\le K\delta(F)$. The regularization lemma of Erd\H{o}s and Simonovits \cite{1970Erdos} is widely used in bipartite Tur\'{a}n problems. We will use the following known variant.

\begin{lemma}[\cite{2012SIAMJiang}]\label{lem:JiangSeiverK-almost}
    Let $\varepsilon,c$ be positive reals, where $\varepsilon<1$ and $c\geq 1$. Let $n$ be a positive integer that is sufficient large as a function of $\varepsilon$. Let $G$ be a graph on $n$ vertices with $e(G)\ge cn^{1+\varepsilon}.$ Then $G$ contains a $K$-almost regular subgraph $G'$ on $m\ge n^{\frac{\varepsilon-\varepsilon^2}{2+2\varepsilon}}$ vertices such that $e(G')\ge \frac{2c}{5}m^{1+\varepsilon}$ and $K= 20\cdot 2^{\frac{1}{\varepsilon^2}+1}$. 
\end{lemma}

In the next result, we write $\hom(H,G)$ for the number of graph homomorphisms from $H$ to $G$. The following lemma is due to Erd\H os and Simonovits.

\begin{lemma}[\cite{ES82}] \label{lem:path inequality}
    Let $k>\ell$ be positive integers such that $k$ is even. Then for any $n$-vertex graph $G$, we have
    $(\frac{\hom(P_{k+1},G)}{n})^{1/k}\geq (\frac{\hom(P_{\ell+1},G)}{n})^{1/\ell}$.
\end{lemma}

A homomorphic even cycle $C_{2\ell}$ in a graph $G$ is a $2\ell$-tuple $(x_{1},x_{2},\ldots,x_{2\ell})\in V(G)^{2\ell}$ such that $x_{1}x_{2}, x_{2}x_{3},\ldots,x_{2\ell}x_{1}\in E(G)$. The following lemma provides a sufficient condition for the existence of homomorphic even cycles $C_{2\ell}$ without ``conflicting" vertices.

\begin{lemma}[\cite{2021arxivJanzer}]\label{lem:UpperBoundHomC2k}
  Let $k\geqslant 2$ be an integer and $G=(V,E)$ be an $n$-vertex non-empty graph. Let $\sim$ be a symmetric binary relation defined over $V$ such that for every $u\in V$ and $v\in V$, $v$ has at most $\beta d(v)$ neighbors $w\in V$ which satisfy $u\sim w$. If $\beta<(2^{20}k^{3}(\log{n})^{4}n^{\frac{1}{k}})^{-1}$, then there exists a homomorphic $2k$-cycle $(x_{1},x_{2},\ldots,x_{2k})$ in $G$ such that for all $i\neq j$, we have $x_{i}\nsim x_{j}$.
\end{lemma}

We will also use the following quantitative version of the supersaturation of even cycles, recently proved by Kim, Lee, Liu and Tran~\cite{2022rainbow}.

\begin{lemma}[\cite{2022rainbow}]\label{Supersaturation of cycles}
    Any $n$-vertex graph $G$ with average degree $d\ge 2\cdot 10^{5}k^{3}n^{1/k}$ contains at least $\frac{1}{2}(2^{12}k)^{-k}d^{2k}$ copies of $2k$-cycles.
\end{lemma}

The following lemma allows us to find a copy of the Cartesian product $P_{t}^{\square}=P_{t}\square K_2$ in a suitable asymmetric bipartite graph, see Figure~\ref{figure:K_2 path}. 

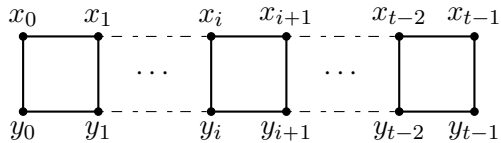
\begin{figure}[htbp]
\centering
\begin{tikzpicture}[scale=1]

\foreach \m in {0,1,2.5,3.5,5,6}{
\foreach \n in {0,1}{
    \node[shape = circle,draw = black,fill,inner sep=0pt,minimum size=1.0mm] at (\m,\n) {};
}
\draw[thick, fill=blue, fill opacity=0.3] (\m,0) -- (\m,1);
}

\node[above] at (0,1) {$x_0$};
\node[above] at (1,1) {$x_1$};
\node[above] at (2.5,1) {$x_i$};
\node[above] at (3.5,1) {$x_{i+1}$};
\node[above] at (5,1) {$x_{t-2}$};
\node[above] at (6,1) {$x_{t-1}$};
\node[below] at (0,0) {$y_0$};
\node[below] at (1,0) {$y_1$};
\node[below] at (2.5,0) {$y_i$};
\node[below] at (3.5,0) {$y_{i+1}$};
\node[below] at (5,0) {$y_{t-2}$};
\node[below] at (6,0) {$y_{t-1}$};

\node at (1.75,0.5) {$\cdots$};
\node at (4.25,0.5) {$\cdots$};

\foreach \n in {0,1}{
\draw[thick, fill opacity=0.3] (0,\n) -- (1,\n);
\draw[dash pattern=on 2pt off 3pt on 4pt off 4pt, fill opacity=0.3] (1,\n) -- (2.5,\n);
\draw[thick, fill opacity=0.3] (2.5,\n) -- (3.5,\n);
\draw[dash pattern=on 2pt off 3pt on 4pt off 4pt, fill opacity=0.3] (3.5,\n) -- (5,\n);
\draw[thick, fill opacity=0.3] (5,\n) -- (6,\n);
}
\end{tikzpicture}
\caption{$P_{t}^{\square}$}
\label{figure:K_2 path}
\end{figure}

\begin{lemma}\label{lem:p2timespt}
Let $t\in\mathbb{N}$ and $C=20t$. Suppose that $H$ is a bipartite graph with parts $X$ and $Y$ such that $e(H)\geq C|Y|$ and $d(x)\geq C|Y|^{1/2}$ for every $x\in X$. Then $H$ contains $P_{t}^{\square}$ as a subgraph.
\end{lemma}

\begin{proof}
It suffices to show that $H$ has a non-empty subgraph $H'$ such that for every edge $xy\in E(H')$, $y$ has at least $2t$ neighbors $z\in V(H')$ with $d_{H'}(x,z)\geq 2t$. Indeed, if such a subgraph exists, then one can embed $P_{t}^{\square}$ by greedily attaching $4$-cycles along an edge one by one.

Define a sequence $H=H_0\supseteq H_1\supseteq H_2\supseteq \cdots $ of graphs as follows. Having defined $H_i$, if there is a vertex $y\in Y$ with $1\leq d_{H_i}(y)\leq \frac{e(H)}{4|Y|}$, then let $H_{i+1}=H_{i}-y$, and call this deletion type 1. If no such vertex exists but there is an edge $e=xy\in E(H_i)$ for which $y\in Y$ has less than $\frac{e(H)}{8|Y|}$ neighbors $z$ in $H_i$ satisfying $d_{H_i}(x,z)\geq 2t$, then let $H_{i+1}=H_{i}-e$, and call this deletion type 2. If no such vertex or edge exists, then set $H'=H_i$ and terminate the process.

Now for any $xy\in E(H')$ with $y\in Y$, it follows immediately from the definition that $y$ has at least $\frac{e(H)}{8|Y|}\geq \frac{C}{8}\geq 2t$ neighbors $z$ in $H'$ such that $d_{H'}(x,z)\geq 2t$. It remains to check that $H'$ is non-empty.

We shall prove that in total at most $\frac{e(H)}{4}$ edges are deleted by type 1 deletions and at most $\frac{e(H)}{4}$ edges are deleted by type 2 deletions. Indeed, there are at most $|Y|$ type 1 deletion steps and each of them removes at most $\frac{e(H)}{4|Y|}$ edges, so it is clear that at most $\frac{e(H)}{4}$ edges are removed during type 1 deletions.

Since in each type 2 deletion step, we remove precisely one edge, it suffices to prove that there are at most $\frac{e(H)}{4}$ such deletion steps throughout the process. Assume that the edge $xy$ gets deleted from $H_i$ because $y$ has less than $\frac{e(H)}{8|Y|}$ neighbors $z$ in $H_i$ such that $d_{H_i}(x,z)\geq 2t$. Since no deletion of type 1 was applied to $H_i$, we have $d_{H_i}(y)\ge\delta(H_{i})>\frac{e(H)}{4|Y|}$. It follows that $y$ has more than $\frac{e(H)}{8|Y|}$ neighbors $z$ in $H_i$ such that $d_{H_i}(x,z)< 2t$. For each such $z$, $y$ is a common neighbor of $x$ and $z$, so $d_{H_{i+1}}(x,z)=d_{H_i}(x,z)-1$. However, the condition $d_{H_i}(x,z)<2t$ infers that any pair $(x,z)$ of vertices in $X$ can be involved in at most $2t$ such type~2 deletions. Let $m$ be the number of edges deleted in type~2 deletions. Via double counting the number of such triples $(x,y,z)$, we have
\begin{equation*}
    m\cdot\frac{e(H)}{8|Y|}\le 2t\cdot |X|^{2}.
\end{equation*}
Thus, there are at most $\frac{16t|X|^2|Y|}{e(H)}$ edges deleted by type 2 deletions. As $e(H)\geq |X|\cdot C|Y|^{1/2}$, we have $\frac{16t|X|^2|Y|}{e(H)}\leq \frac{e(H)}{4}$, completing the proof.
\end{proof}

As a corollary of Lemma \ref{lem:p2timespt}, we have the following result.

\begin{lemma} \label{lem:large codegree rare}
    For any $K$ there is some $C_0=C_0(K,\ell)$ such that the following holds. Let $G$ be an $n$-vertex $C_{2\ell}^{\square}$-free bipartite graph with average degree $d\geq C_0n^{1/2}$ and maximum degree at most $Kd$, and let $uv\in E(G)$. Then the number of $4$-cycles $uvwz$ with $d(u,w)>C_0d^{1/2}$ is at most $C_0d$.
\end{lemma}

\begin{proof}
We will show that $C_0=100K \ell$ is suitable. Let $X=\{w\in N(v): d_G(u,w)> C_0d^{1/2}\}\setminus \{u\}$ and let $Y=N(u)\setminus \{v\}$. Since $G$ is bipartite, $X$ and $Y$ are disjoint sets. Define a bipartite graph $H$ with parts $X$ and $Y$ in which there is an edge between $x\in X$ and $y\in Y$ if and only if $xy$ is an edge in $G$. Now observe that for each $x\in X$, we have $d_H(x)\geq d_G(u,x)-1>C_0d^{1/2}/2\geq C_0 K^{-1/2} |Y|^{1/2}/2\geq 50\ell |Y|^{1/2}$. Moreover, $H$ does not contain $P_{2\ell-1}^{\square}$ as a subgraph since such a subgraph could be extended to a $C_{2\ell}^{\square}$ in $G$, using the additional vertices $u$ and $v$. Hence, by Lemma \ref{lem:p2timespt}, we must have $e(H)<40\ell|Y|$. This implies that there are at most $40\ell|Y|\leq 40\ell Kd\leq C_0d$ cycles $uvwz$ with $d_G(u,w)>C_0d^{1/2}$.
\end{proof}

\section{Grid, cylinder and torus}\label{sec:Grids}
In this section, we will prove Theorems~\ref{thm:normal grid}, ~\ref{thm:grid} and~\ref{thm:torus}. The proofs of all these results utilize a strategy that reduces the embedding problems to finding a collection of paths or cycles with a certain nice property.

For grids, we will need a collection of paths.

\begin{defn}
Let $\alpha>0$ and $k\in\mathbb{N}$. We say that a collection $\mathcal{P}$ of (labelled) paths $P_{k}$ is \emph{$\alpha$-rich} if for any member $x_1x_2\cdots x_k\in\mathcal{P}$ and any $2\leq i\leq k-1$, there exist at least $\alpha$ distinct vertices $x_i'$ such that $x_1x_2\cdots x_{i-1}x_i'x_{i+1}\cdots x_k\in \mathcal{P}$.
\end{defn}

For the quadrangulations of the cylinder and the torus, we will need a collection of cycles.

\begin{defn}
Let $\alpha>0$ and $\ell\in\mathbb{N}$ with $\ell\ge 2$. We say that a collection $\mathcal{C}$ of (labelled) cycles $C_{2\ell}$ is \emph{$\alpha$-rich} if for any member $x_1x_2\cdots x_{2\ell}x_1\in\mathcal{C}$ and any $1\leq i\leq 2\ell$, there exist at least $\alpha$ distinct vertices $x_i'$ such that $x_1x_2\cdots x_{i-1}x_i'x_{i+1}\cdots x_{2\ell}x_1\in\mathcal{C}$.
\end{defn}

\begin{prop}[Finding a grid from a rich collection of paths]\label{prop:normal grid}
    Let $\mathcal{P}$ be a non-empty $\alpha$-rich collection of paths of length $2t-2$ in $G$.
    If $\alpha \ge  t^2$, then $G$ contains a copy of $F_{t,t}$ as a subgraph. 
\end{prop}

\begin{proof}
    Let $Q_0=x_{1,1}x_{1,2}\cdots x_{1,t}x_{2,t}\cdots x_{t,t}$ be a path in $\mathcal{P}$. Since $\mathcal{P}$ is $\alpha$-rich, we can replace $x_{1,t}$ by a different vertex $x_{2,{t-1}}$ and get another path $Q_1=x_{1,1}x_{1,2}\cdots x_{1,t-1}x_{2,t-1}x_{2,t}x_{3,t}\cdots x_{t,t}$ in $\mathcal{P}$. Again, since $\mathcal{P}$ is $\alpha$-rich and $Q_1\in \mathcal{P}$, we can replace $x_{1,t-1}$ by a vertex $x_{2,t-2}$ that is different from all the previous vertices and get another path $Q_2=x_{1,1}x_{1,2}\cdots x_{1,t-2}x_{2,t-2}x_{2,t-1}x_{2,t}x_{3,t}\cdots x_{t,t}$ in $\mathcal{P}$. We may continue like this and eventually get a path $Q_{t-1}=x_{1,1}x_{2,1}x_{2,2}\cdots x_{2,t}x_{3,t}\cdots x_{t,t}$ in $\mathcal{P}$. Then we may replace $x_{2,t}$ by a vertex $x_{3,t-1}$ that is different from all previous vertices, and get a path $Q_t=x_{1,1}x_{2,1}x_{2,2}x_{2,3}\cdots x_{2,t-1}x_{3,t-1}x_{3,t}x_{4,t}\cdots x_{t,t}$ in $\mathcal{P}$. Continuing this process in the obvious way, we end up with a path $x_{1,1}x_{2,1}\cdots x_{t,1}x_{t,2}\cdots x_{t,t}$ in $\mathcal{P}$. The vertices $x_{i,j}$ ($i,j\in [t]$) defined through this process form a $t\times t$ grid~(e.g. for $t=4$, see Figure~\ref{grid:t=4}).
\end{proof}

\begin{figure}[h]
    \centering
    \includegraphics[width=12cm]{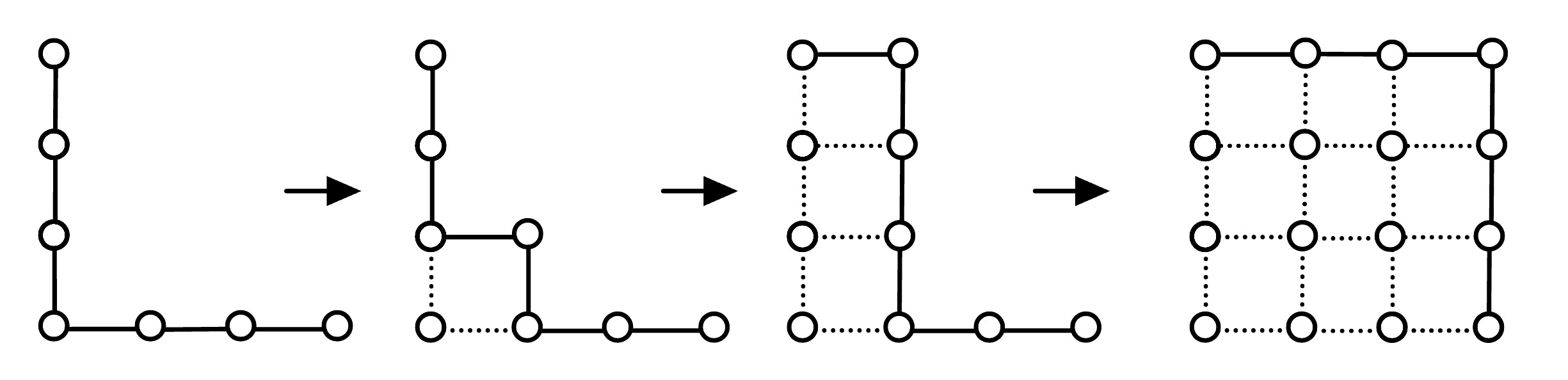}
    \caption{The process in Proposition~\ref{prop:normal grid} for $t=4$}
    \label{grid:t=4}
\end{figure}

\begin{prop}[Finding a quadrangulated cylinder from a rich collection of cycles]\label{prop:45Grid}
    Let $\mathcal{C}$ be a non-empty $\alpha$-rich collection of $2\ell$-cycles in $G$.
    If $\alpha \ge  k\ell$, then $G$ contains a copy of $P_{k,\ell}$ as a subgraph. 
\end{prop}
\begin{proof}
    We find a copy of $P_{k,\ell}$ as follows. First, take an arbitrary $2\ell$-cycle $C_{0}\in\mathcal{C}$ with $C_{0}=a_{1}b_{1}a_{2}b_{2} \cdots a_{\ell}b_{\ell}a_{1}$. 
    Since $\mathcal{C}$ is $\alpha$-rich, and $\alpha \ge k\ell\ge V(C_{0})$, there exist $C_{1}\in\mathcal{C}$ and $c_{\ell}\notin V(C_{0})$ such that $C_{0}-a_{1}=C_{1}-c_{\ell}$. Then, for $i=1,2,\ldots,\ell-1$, there exist $C_{i+1}\in\mathcal{C}$ and $c_{i}\notin\bigcup_{j=0}^{i}V(C_{j})$ such that $C_{i}-a_{i+1}=C_{i+1}-c_{i}$. The union of the cycles $C_0,C_1,\dots,C_{\ell}$ gives a copy of $P_{3,\ell}$ (e.g. for $\ell=3$, see Figure~\ref{alphaP33}). Repeating the same process, as long as $\alpha \ge k\ell$, we can greedily find a copy of $P_{k,\ell}$.
\end{proof}

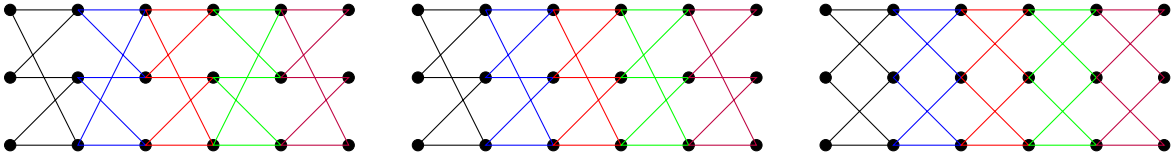
\begin{figure}[h]
\centering
\begin{tikzpicture}[scale=0.9]
\begin{scope}
\foreach \m in {1,2,3,4,5,6}
\foreach \n in {1,2,3}{
\node[shape = circle,draw = black,fill=black,inner sep=0pt,minimum size=1.5mm] at (\m,\n) {};
}

\draw[black] (1,3) -- (2,3) -- (1,2) -- (2,2) --(1,1) -- (2,1) -- (1,3);
\draw[blue] (2,3) -- (3,2) -- (2,2) -- (3,1) --(2,1) -- (3,3) -- (2,3);
\draw[red] (3,3) -- (4,3) -- (3,2) -- (4,2) --(3,1) -- (4,1) -- (3,3);
\draw[green] (4,3) -- (5,2) -- (4,2) -- (5,1) --(4,1) -- (5,3) -- (4,3);
\draw[purple] (5,3) -- (6,3) -- (5,2) -- (6,2) --(5,1) -- (6,1) -- (5,3);

\end{scope}
\end{tikzpicture}
\hspace{5mm}
\begin{tikzpicture}[scale=0.9]
\begin{scope}
\foreach \m in {1,2,3,4,5,6}
\foreach \n in {1,2,3}{
\node[shape = circle,draw = black,fill=black,inner sep=0pt,minimum size=1.5mm] at (\m,\n) {};
}

\draw[black] (1,3) -- (2,3) -- (1,2) -- (2,2) --(1,1) -- (2,1) -- (1,3);
\draw[blue] (2,3) -- (3,3) -- (2,2) -- (3,2) --(2,1) -- (3,1) -- (2,3);
\draw[red] (3,3) -- (4,3) -- (3,2) -- (4,2) --(3,1) -- (4,1) -- (3,3);
\draw[green] (4,3) -- (5,3) -- (4,2) -- (5,2) --(4,1) -- (5,1) -- (4,3);
\draw[purple] (5,3) -- (6,3) -- (5,2) -- (6,2) --(5,1) -- (6,1) -- (5,3);

\end{scope}
\end{tikzpicture}
\hspace{5mm}
\begin{tikzpicture}[scale=0.9]
\begin{scope}
\foreach \m in {1,2,3,4,5,6}
\foreach \n in {1,2,3}{
\node[shape = circle,draw = black,fill=black,inner sep=0pt,minimum size=1.5mm] at (\m,\n) {};
}

\draw[black] (1,3) -- (2,3) -- (1,2) -- (2,1) --(1,1) -- (2,2) -- (1,3);
\draw[blue] (3,3) -- (2,3) -- (3,2) -- (2,1) --(3,1) -- (2,2) -- (3,3);
\draw[red] (3,3) -- (4,3) -- (3,2) -- (4,1) --(3,1) -- (4,2) -- (3,3);
\draw[green] (5,3) -- (4,3) -- (5,2) -- (4,1) --(5,1) -- (4,2) -- (5,3);
\draw[purple] (5,3) -- (6,3) -- (5,2) -- (6,1) --(5,1) -- (6,2) -- (5,3);
\end{scope}
\end{tikzpicture}
\caption{Three ways to draw the graph $P_{6,3}$}
\end{figure}

\begin{figure}[h]
    \centering
    \includegraphics[width=11.8cm]{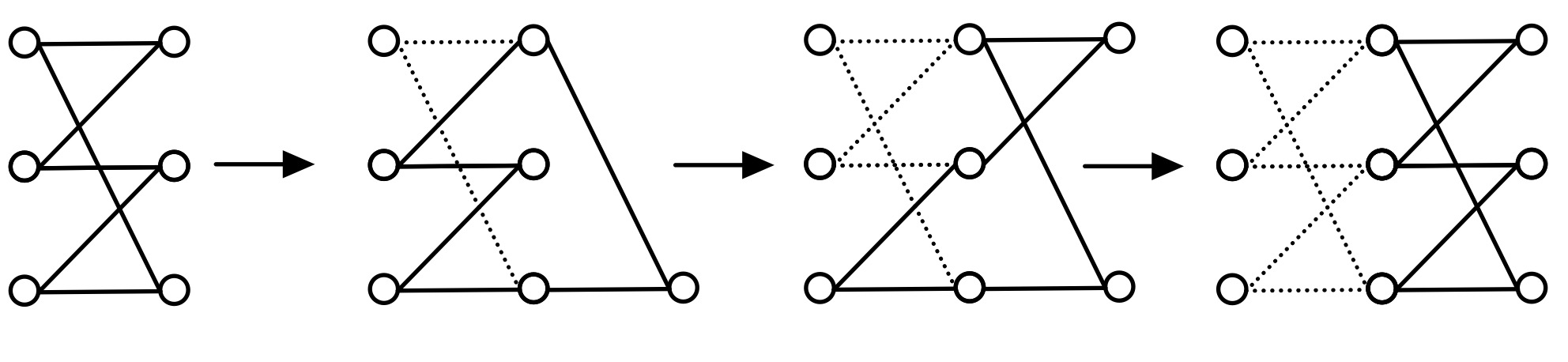}
    \caption{The process in Proposition~\ref{prop:45Grid} for $\ell=3$}
    \label{alphaP33}
\end{figure}

\begin{prop}[Finding a quadrangulated torus from a rich collection of cycles]\label{prop:torus}
    Let $k$ be an even integer, let $G$ be an $n$-vertex bipartite graph and let $\mathcal{C}$ be a non-empty $\alpha$-rich collection of $2\ell$-cycles in $G$.
    If $\frac{\alpha}{\ell^{2}}>  2^{20}(\frac{k}{2})^{3}(\ell\log{n})^{4}n^{2\ell/k}$, then $G$ contains a copy of $T_{k,\ell}$ as a subgraph. 
\end{prop}
\begin{proof}
    Let $A\cup B$ be a bipartition of $G$. We define a bipartite auxiliary graph $H$ as follows. The vertex set of $H$ is $A^{\ell}\cup B^{\ell}$ and two vertices $\boldsymbol{x}=(x_1,x_2,\ldots,x_\ell)\in A^{\ell}$ and $\boldsymbol{y}=(y_1,y_2,\ldots,y_\ell)\in B^{\ell}$ are adjacent in $H$ if $x_{1}y_{1}x_{2}y_{2}\cdots x_{\ell}y_{\ell}x_{1} \in \mathcal{C}$.
    
    We claim that for any $\boldsymbol{x},\boldsymbol{y}\in V(H)$, the number of $\boldsymbol{z}=(z_{1},z_{2},\ldots,z_{\ell})\in N_H(\boldsymbol{y})$ such that there exist some $i,j\in[\ell]$ with $z_i=x_j$ is at most $ \frac{\ell^2}{\alpha} d_H(\boldsymbol{y})$. To see this, let $T:=\{\boldsymbol{z}\in N_H(\boldsymbol{y})$: there exist $i,j\in[\ell]$ with $z_i=x_j\}$ and $T_{i,j}:=\{\boldsymbol{z}\in N_H(\boldsymbol{y})$: $z_i=x_j\}$. Clearly $|T|\le \sum_{i,j=1}^{\ell}|T_{i,j}|$. Since $\mathcal{C}$ is $\alpha$-rich, for each $\boldsymbol{z}\in T_{i,j}$, we have at least $\alpha$ distinct choices of $\boldsymbol{z}'$ such that $\boldsymbol{z}'=(z_{1}',z_{2}',\ldots,z_{\ell}')\in N_H(\boldsymbol{y})$ and $z_{r}'=z_{r}$ with $r\in [\ell]\setminus\{i\}$. Moreover, all of these $\boldsymbol{z}'$ corresponding to different $\boldsymbol{z}$ are distinct, which implies $d_H(\boldsymbol{y}) \ge \alpha |T_{i,j}|$. Therefore we have $\ell^2d_H(\boldsymbol{y})\ge \alpha |T|$, then the claim follows.

    Define a binary relation $\sim$ over $V(H)$ by setting $\boldsymbol{x}\sim \boldsymbol{y}$ if and only if there exist some $i,j\in[\ell]$ with $x_i=y_j$. Since $|V(H)| \leq n^\ell$, by the above claim and Lemma~\ref{lem:UpperBoundHomC2k} with $\beta=\ell^2/\alpha$, there exists a $k$-cycle $\boldsymbol{a}_1\boldsymbol{a}_2\cdots \boldsymbol{a}_k\boldsymbol{a}_1$ in $H$ with $\boldsymbol{a}_i\nsim \boldsymbol{a}_j$ for any different $i,j\in [k]$, which gives a copy of $T_{k,\ell}$ in $G$.   
\end{proof}

It remains to show that rich collection of paths and cycles with suitable parameters can be found in our host graphs. Combined with Proposition \ref{prop:normal grid}, the next lemma proves Theorem \ref{thm:normal grid}.

\begin{lemma}
    Let $t$ be a positive integer, let $n$ be sufficiently large compared to $t$ and let $G$ be an $n$-vertex graph with  $e(G)\geq 5t^{3/2}n^{3/2}$. Then $G$ has a non-empty $\alpha$-rich collection of paths of length $2t-2$ with $\alpha=t^2$.
\end{lemma}

\begin{proof}
    By Lemma \ref{lem:JiangSeiverK-almost}, $G$ has a $K$-almost regular subgraph $G'$ on with $m$ vertices and at least $2t^{3/2}m^{3/2}$ edges, where $K$ is some absolute constant and $m$ is sufficiently large compared to $t$. Let $d\geq 4t^{3/2}m^{1/2}$ be the average degree of $G'$. Let $\alpha=t^2$. Let $\mathcal{P}_0$ be the collection of paths of length $2t-2$ in $G'$. Since $G'$ is $K$-almost regular, it is easy to see that more than half of all homomorphisms from $P_{2t-1}$ to $G'$ are injective. Hence, we have $|\mathcal{P}_0|\geq \frac{1}{2}\hom(P_{2t-1},G')$.
    
    Define a sequence of collections of paths of length $2t-2$ in $G'$, $\mathcal{P}_0 \supseteq \mathcal{P}_1 \supseteq \cdots$. Having defined $\mathcal{P}_i$, if there exist some $P\in \mathcal{P}_i$ and an internal vertex $a\in P$ such that the number of $P'\in \mathcal{P}_i$ containing some $b\in P'$ with $P'-b = P-a$ is less than $\alpha$, then define $F_{i}:=P-a$ and let $\mathcal{P}_{i+1}$ be the collection obtained from $\mathcal{P}_i$ by removing all paths of length $2t-2$ containing $F_{i}$. The process will terminate if there is no such member $P$.

    Suppose that the process stops after $s$ steps, it suffices to show that $\mathcal{P}_{s}$ is non-empty. As for each $i\in \{0,1,\ldots,s-1\}$, all paths of length $2t-2$ in $\mathcal{P}_{i}$ containing $F_{i}$ were removed, all the $F_{i}$'s are distinct. Hence the number of possible steps $s$ is at most the number of different possibilities for $F_i$.

    \begin{claim}
        There are at most $\frac{1}{4t^2}\hom(P_{2t-1},G')$ possibilities for $F_i$.
    \end{claim}
\begin{poc}
 There are less than $2t$ choices for the position of the removed vertex, so it suffices to prove that for each $0\leq \ell\leq 2t-4$, the number of subgraphs in $G'$ isomorphic to $P_{\ell+1}\cup P_{2t-3-\ell}$ is at most $\frac{1}{8t^3}\hom(P_{2t-1},G')$. Clearly, the number of such subgraphs in $G'$ is at most $\hom(P_{\ell+1},G')\hom(P_{2t-3-\ell},G')$. By Lemma \ref{lem:path inequality}, we have
    \begin{align*}
        \hom(P_{\ell+1},G')\hom(P_{2t-3-\ell},G')
        &\leq m\left(\frac{\hom(P_{2t-1},G')}{m}\right)^{\ell/(2t-2)}m\left(\frac{\hom(P_{2t-1},G')}{m}\right)^{(2t-4-\ell)/(2t-2)} \\
        &=m^2\left(\frac{\hom(P_{2t-1},G')}{m}\right)^{(2t-4)/(2t-2)}=m^{\frac{t}{t-1}}\hom(P_{2t-1},G')^{(t-2)/(t-1)}
    \end{align*}
    By Sidorenko's property for paths~\cite{1959SidorenkoPath}, we have $\hom(P_{2t-1},G')\geq md^{2t-2}$, so $\hom(P_{2t-1},G')^{1/(t-1)}\geq m^{1/(t-1)}d^2\geq 16t^3m^{t/(t-1)}$. Hence, $m^{\frac{t}{t-1}}\hom(P_{2t-1},G')^{(t-2)/(t-1)}\leq \frac{1}{16t^3}\hom(P_{2t-1},G')$, completing the proof of the claim.   
\end{poc}
    By the claim and our earlier discussion, it follows that $s\leq \frac{1}{4t^2}\hom(P_{2t-1},G')$.
    Hence, we have 
    \begin{equation*}
        |\mathcal{P}_{s}| \ge |\mathcal{P}_0|-s\cdot \alpha \ge  \frac{1}{2}\hom(P_{2t-1},G')-  st^2 >0,
    \end{equation*}
     as desired.
\end{proof}

For Theorem~\ref{thm:grid}, by Proposition~\ref{prop:45Grid}, it suffices to find an $\alpha$-rich collection of $2\ell$-cycles with $\alpha\ge k\ell$. 
For Theorem~\ref{thm:torus}, by Proposition~\ref{prop:torus}, it suffices to find an $\alpha$-rich collection of $2\ell$-cycles with $\alpha\ge \ell^2\cdot 2^{20}(\frac{k}{2})^{3}(\ell\log{n})^{4}n^{2\ell/k}$. Thus, both theorems follow from the following lemma.

\begin{lemma}\label{lemma:alphaNice}
    Let $\ell\geq 2$ be an integer, let $n$ be sufficiently large compared to $\ell$ and let $G$ be a graph on $n$ vertices with $e(G)\ge Cn^{3/2}$, where $C\geq 1$. Then $G$ contains a non-empty $\alpha$-rich collection of $C_{2\ell}$, where $\alpha = \frac{c\cdot C^2}{2^{2\ell}K^{2\ell-2}}$, $c=\frac{1}{2}(2^{12}\ell)^{-\ell}$ and $K= 20\cdot 2^5$.
\end{lemma}

\begin{proof}
    By Lemma~\ref{lem:JiangSeiverK-almost}, it suffices to prove the result for graphs $G$ which are $K$-almost regular and have $e(G) = \frac{2}{5}Cn^{3/2}$. Let $\mathcal{C}_{0}$ be the collection of all $2\ell$-cycles in $G$. By Lemma~\ref{Supersaturation of cycles}, we know $|\mathcal{C}_{0}| \ge c \cdot(\frac{4}{5}C n^{1/2})^{2\ell}$, where $c=\frac{1}{2}(2^{12}\ell)^{-\ell}$.

Define a sequence of collections of $2\ell$-cycles, $\mathcal{C}_0 \supseteq \mathcal{C}_1 \supseteq \cdots$ as follows. Having defined $\mathcal{C}_i$, if there exist some $C\in \mathcal{C}_i$ and a vertex $a\in C$ such that the number of $C'\in \mathcal{C}_i$ containing some $b\in C'$ with $C'-b = C-a$ is less than $\alpha$, then define $F_{i}:=C-a$ and let $\mathcal{C}_{i+1}$ be the collection obtained from $\mathcal{C}_i$ by removing all $2\ell$-cycles containing the $(2\ell-1)$-vertex path $F_{i}$. The process will terminate if there is no such member $C$.
    
    Suppose that the process stops after $t$ steps, it suffices to show that $\mathcal{C}_{t}$ is non-empty. As for each $i\in \{0,1,\ldots,t-1\}$, all cycles in $\mathcal{C}_{i}$ containing $F_{i}$ were removed, all of $F_{i}$'s are distinct. Hence the number of possible steps $t$ is at most the number of $P_{2\ell-1}$ in $G$, which implies $t\le n\cdot\Delta(G)^{2\ell-2}\leq n\cdot (KCn^{1/2})^{2\ell-2}$.
    Then we have 
    \begin{equation*}
        |\mathcal{C}_{t}| \ge |\mathcal{C}_0|-t\cdot \alpha \ge  c \cdot\left(\frac{4}{5} Cn^{1/2}\right)^{2\ell}-  t \alpha \geq \left(c\cdot \left(\frac{4C}{5}\right)^{2\ell} - \alpha \cdot K^{2\ell-2} C^{2\ell-2}\right)n^\ell >0,
    \end{equation*}
     as desired.
\end{proof}

\section{Honeycomb: Proof of Theorem~\ref{thm:tightHoneyComb}}\label{sec:honeycomb}
With the help of the following notion of a \emph{good} collection of paths, we can implement a similar shifting embedding strategy for honeycomb graphs. However, constructing such a good collection is substantially more involved due to the many possible degenerate cases. Instead of using supersaturation, we carry out a weighted count for paths.

\begin{defn}\label{def:alphagood}
    Let $\alpha>0$ and $k\in \mathbb{N}$. A collection $\mathcal{P}$ of paths $P_{k}$ is \emph{$\alpha$-good} if the following holds. For any $x_1x_2\cdots x_{k}\in \mathcal{P}$ and $2\leq i\leq k-2$, there are at least $\alpha$ pairwise disjoint edges $x'_ix'_{i+1}$ such that $x_1x_2\cdots x_{i-1}x'_ix'_{i+1}x_{i+2}\cdots x_{k}\in \mathcal{P}$.
\end{defn}

Similarly to Proposition~\ref{prop:normal grid}, we can show that the existence of a suitable $\alpha$-good collection of paths guarantees the existence of a large honeycomb subgraph.
\begin{prop}[Finding a honeycomb $H_{k,\ell}$ from a good collection of paths]\label{prop:Honeycomb}
    Let $\mathcal{C}$ be a non-empty $\alpha$-good collection of paths $P_{2k}$ in $G$.
    If $\alpha \ge  k\ell$, then $G$ contains a copy of $H_{k,\ell}$ as a subgraph. 
\end{prop}
\begin{figure}[h]
    \centering
    \includegraphics[width=10cm]{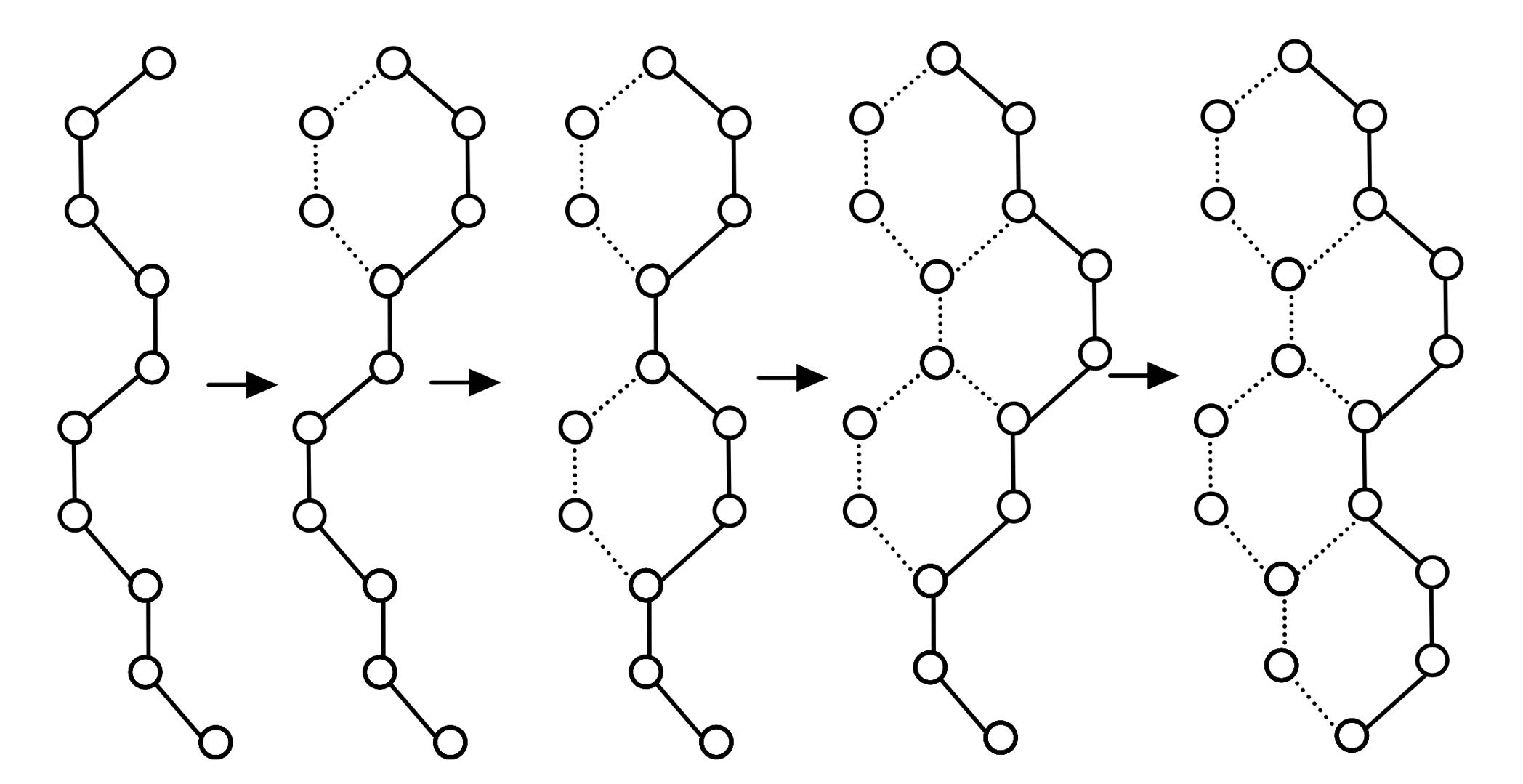}
    \caption{The process in Proposition~\ref{prop:Honeycomb} for $k=5$}
    \label{k=10}
\end{figure}

Instead of giving a formal proof of Proposition~\ref{prop:Honeycomb}, we refer the reader to Figure~\ref{k=10}, which depicts the embedding procedure. The process starts with an arbitrary member of an $\alpha$-good collection of paths $P_{2k}$, and then in each step we replace two (suitably chosen) consecutive vertices of the path in a way that the new path still belongs to the $\alpha$-good collection, and the new vertices are distinct from all previous vertices appearing in the process. By taking enough steps, the graph induced by the vertices appearing in the process will contain $H_{k,\ell}$.

By Proposition~\ref{prop:Honeycomb}, it suffices to find a non-empty $\alpha$-good collection of paths $P_{2k}$ with $\alpha \ge k\ell$. Note that if $\mathcal{C}$ is an $\alpha$-good collection of paths $P_{2k+1}$, then for any vertex $v$, the collection of those paths $x_1x_2\dots x_{2k}$ for which $x_1x_2\dots x_{2k}v\in \mathcal{C}$ is an $\alpha$-good collection of paths $P_{2k}$. Hence, it suffices to find a non-empty $\alpha$-good collection of paths $P_{2k+1}$ with $\alpha\geq k\ell$.

\begin{lemma} \label{lem:findgoodpaths}
    For any $k\in \mathbb{N}$ and positive real numbers $K$ and $\alpha$, there exists some $C=C(k,K,\alpha)$ such that the following holds. Let $G$ be an $n$-vertex $K$-almost-regular graph with $e(G)\geq Cn^{4/3}$. Then $G$ contains a non-empty $\alpha$-good collection of paths $P_{2k+1}$.
\end{lemma}

\begin{proof}
Let $d=2e(G)/n\geq Cn^{1/3}$.
Let $L$ be a constant that is sufficiently large compared to $k$, $K$ and $\alpha$, and let $C$ be sufficiently large compared to $L$. We consider two cases.

\medskip

\noindent\textbf{Case 1.} There are at most $nd^2/L$ paths $uvw$ in $G$ such that $d(u,w)>C$.

Note that the number of paths of length $2k$ in $G$ is at least $\frac{1}{2}n\delta(G)^{2k}\geq \frac{1}{2K^{2k}}nd^{2k}$. On the other hand, the number of paths $x_0x_1\cdots x_{2k}$ such that there exists some $0\leq i\leq 2k-2$ with $d(x_i,x_{i+2})>C$ is at most $2kK^{2k}nd^{2k}/L$. Indeed, there are at most $2k$ ways to choose $i$, there are at most $nd^2/L$ ways to choose $x_ix_{i+1}x_{i+2}$ and there are at most $\Delta(G)^{2k-2}\leq (Kd)^{2k-2}$ ways to extend this to a path of length $2k$. Using the fact that $L$ is sufficiently large compared to $k$ and $K$, it follows that there are at least $\frac{1}{4K^{2k}}nd^{2k}$ paths $x_0x_1\cdots x_{2k}$ in $G$ with $d(x_i,x_{i+2})\leq C$ for all $0\leq i\leq 2k-2$. Let $\cP_0$ be the collection of these paths. 

Now for $j=0,1,\dots$, if there exist some $0\leq i\leq 2k-3$ and some $x_0x_1\cdots x_{2k}\in \cP_j$ such that there are at most $C^2$ edges $x'_{i+1}x'_{i+2}$ with $x_0x_1\cdots x_ix'_{i+1}x'_{i+2}x_{i+3}\cdots x_{2k}\in \cP_j$, then take one such choice of $i$ and $x_0x_1\cdots x_{2k}$, and let $\cP_{j+1}$ be the collection obtained from $\cP_j$ by discarding all members of $\cP_j$ of the form $x_0x_1\cdots x_ix'_{i+1}x'_{i+2}x_{i+3}\cdots x_{2k}$. If no such $i$ and $x_0x_1\cdots x_{2k}\in \cP_j$ exist, then terminate the process. Let $\cP=\cP_t$ be the collection of paths when the process terminates. 

We claim that $\cP$ is non-empty. Indeed, we must have $t\leq 2k\cdot n^2\Delta(G)^{2k-3}$ because there are at most $2k$ choices for $i$ in the above step, there are at most $n^2\Delta(G)^{2k-3}$ choices for the vertices $x_0,x_1,\ldots,x_i,x_{i+3},x_{i+4},\ldots,x_{2k}$, and any pair consisting of $i$ and $(x_0,x_1,\ldots,x_i,x_{i+3},x_{i+4},\ldots,x_{2k})$ features at most once in the above process. Hence, indeed, we have $t\leq 2kn^2\Delta(G)^{2k-3}$. Furthermore, in each step, at most $C^2$ paths are deleted, so we have $|\cP|\geq |\cP_0|-tC^2\geq \frac{1}{4K^{2k}}nd^{2k}-2kn^2\Delta(G)^{2k-3}C^2$, which is greater than $0$ since $d\geq Cn^{1/3}$ and $C$ is sufficiently large compared to $k$ and $K$. 

We shall now prove that $\cP$ is $\alpha$-good, which completes the proof in Case 1. Assuming otherwise, there exist some $i$ and some $x_0\cdots x_{2k}\in \cP$ such that there are fewer than $\alpha$ pairwise disjoint edges $x'_{i+1}x'_{i+2}$ satisfying $x_0\cdots x_ix'_{i+1}x'_{i+2}x_{i+3}\cdots x_{2k}\in \cP$. This means that there is a collection of $s\leq \alpha$ edges $e_1,\ldots,e_s$ in $G$ such that whenever $x_0\cdots x_ix'_{i+1}x'_{i+2}x_{i+3}\cdots x_{2k}\in \cP$, we have that $x'_{i+1}x'_{i+2}$ intersects some $e_j$. Hence, there is a set $S$ of vertices of size at most $2s$ such that whenever $x_0\dots x_ix'_{i+1}x'_{i+2}x_{i+3}\cdots x_{2k}\in \cP$, we have $x'_{i+1}\in S$ or $x'_{i+2}\in S$. But note that if $x_0\cdots x_ix'_{i+1}x'_{i+2}x_{i+3}\cdots x_{2k}\in \cP$, then $d(x_i,x'_{i+2})\leq C$ and $d(x'_{i+1},x_{i+3})\leq C$. This means that there are at most $2\cdot 2s\cdot C$ edges $x'_{i+1}x'_{i+2}$ such that $x_0\cdots x_ix'_{i+1}x'_{i+2}x_{i+3}\cdots x_{2k}\in \cP$. Since $4sC<C^2$, this is a contradiction.

\medskip

\noindent\textbf{Case 2.} There are more than $nd^2/L$ paths $uvw$ in $G$ such that $d(u,w)>C$.

Choose $v\in V(G)$ such that there are more than $d^2/L$ paths of the form $uvw$ such that $d(u,w)>C$. Let $\cQ$ be the collection of paths $x_0x_1\cdots x_{2k}$ of length $2k$ in $G-v$ such that $x_{2i}\in N(v)$ for each $0\leq i\leq k$ and $d(x_{2i-2},x_{2i})>C$ for each $1\leq i\leq k$. Define the weight of such a path $P\in \cQ$ to be $w(P)=1/\prod_{i=1}^k d(x_{2i-2},x_{2i})$.

\medskip

\begin{claim} We have $$\sum_{P\in \cQ} w(P)\geq \Omega_{k,K,L}(d^{k+1}).$$
\end{claim}

\begin{poc} Define an auxiliary graph $H$ with vertex set $N(v)$ in which vertices $u$ and $w$ form an edge if $d(u,w)>C$. By the assumption that there are more than $d^2/L$ pairs of vertices $u,w\in N(v)$ such that $d(u,w)>C$, we have $e(H)>d^2/L$.
Since $|N(v)|\leq Kd$, $H$ has average degree at least $2d/(KL)$. By Lemma~\ref{lem:SubgraphLargeDeg}, $H$ has a subgraph $H'$ with minimum degree at least $d/(2KL)$. Since $d\geq Cn^{1/3}\geq C$ and $C$ is sufficiently large compared to $K$, $L$ and $k$, it follows (by a greedy embedding procedure in $H'$) that there are at least $(d/(4KL))^{k+1}=\Omega_{k,K,L}(d^{k+1})$ copies of $P_{k+1}$ in the graph $H$. Each such copy corresponds to a $(k+1)$-tuple $(x_0,x_2,x_4,\dots,x_{2k})$ consisting of distinct vertices in $N(v)$ such that $d(x_{2i-2},x_{2i})>C$ holds for all $1\leq i\leq k$. The number of ways to extend such a tuple to a member of $\cQ$ is at least $\prod_{i=1}^k (d(x_{2i-2},x_{2i})/2)$ and any such member of $\cQ$ has weight $1/\prod_{i=1}^k d(x_{2i-2},x_{2i})$, completing the proof of the claim. 
\end{poc}

Let $\cP_0=\cQ$. Let us define a sequence $\cP_0\supset \cP_1\supset \cdots$ of collections of paths as follows.

\begin{itemize}
    \item Having defined $\cP_j$, if there exist some $1\leq i\leq k$ and $x_0x_1\cdots x_{2k}\in \cP_j$ such that the number of vertices $x'_{2i-1}$ with $x_0x_1\cdots x_{2i-2}x'_{2i-1}x_{2i}\cdots x_{2k}\in \cP_j$ is at most $2\alpha$, then choose such an $i$ and $x_0x_1\cdots x_{2k}\in \cP_j$, and let $\cP_{j+1}$ be obtained from $\cP_j$ by deleting all elements of the form $x_0x_1\cdots x_{2i-2}x'_{2i-1}x_{2i}\cdots x_{2k}$. Call this deletion of type 1.

    \item Otherwise, if there exist some  $0\leq i\leq k-2$ and $x_0x_1\cdots x_{2k}\in \cP_j$ such that there are at most $2\alpha$ vertices $x'_{2i+2}$ for which there exists a vertex $x'_{2i+1}$ such that $x_0x_1\cdots x_{2i}x'_{2i+1}x'_{2i+2}x_{2i+3}\cdots x_{2k}\in \cP_j$, then choose such an $i$ and $x_0x_1\cdots x_{2k}\in \cP_j$, and let $\cP_{j+1}$ be obtained from $\cP_j$ by deleting all paths of the form $x_0x_1\dots x_{2i}x'_{2i+1}x'_{2i+2}x_{2i+3}\cdots x_{2k}$. Call this deletion of type 2.

    \item Otherwise, if there exist some  $1\leq i\leq k-1$ and $x_0x_1\cdots x_{2k}\in \cP_j$ such that there are at most $2\alpha$ vertices $x'_{2i}$ for which there exists a vertex $x'_{2i+1}$ such that $x_0x_1\cdots x_{2i-1}x'_{2i}x'_{2i+1}x_{2i+2}\cdots x_{2k}\in \cP_j$, then choose such an $i$ and $x_0x_1\cdots x_{2k}\in \cP_j$, and let $\cP_{j+1}$ be obtained from $\cP_j$ by deleting all paths of the form $x_0x_1\cdots x_{2i-1}x'_{2i}x'_{2i+1}x_{2i+2}\cdots x_{2k}$. Call this deletion of type 3.

    \item Else, terminate the process.
\end{itemize}

Let $\cP$ be the collection of paths when the process terminates. Observe that $\cP$ is $\alpha$-good. Indeed, we just need to show that whenever $0\leq h\leq 2k-3$ and $x_0x_1\cdots x_{2k}\in \cP$, then there are at least $\alpha$ pairwise disjoint edges $x'_{h+1}x'_{h+2}$ such that $x_0x_1\cdots x_h x'_{h+1}x'_{h+2} x_{h+3}\cdots x_{2k}\in \cP$. Let us assume that $h=2i$ is even (the odd case is basically identical). Since the process terminated with $\cP$, there are at least $2\alpha$ vertices $x'_{2i+2}$ for which there exists a vertex $x''_{2i+1}$ such that $x_0x_1\cdots x_{2i}x''_{2i+1}x'_{2i+2}x_{2i+3}\cdots x_{2k}\in \cP$. However, (again since the process terminated with $\cP$) given such a vertex $x'_{2i+2}$, there must exist at least $2\alpha$ vertices $x'_{2i+1}$ satisfying $x_0x_1\cdots x_{2i}x'_{2i+1}x'_{2i+2}x_{2i+3}\cdots x_{2k}\in \cP$. This clearly implies that there are at least $\alpha$ pairwise disjoint edges $x'_{2i+1}x'_{2i+2}$ such that $x_0x_1\cdots x_{2i}x'_{2i+1}x'_{2i+2}x_{2i+3}\cdots x_{2k}\in \cP$.

It remains to prove that $\cP$ is nonempty. To show this, we will argue that the total weight of paths deleted by the process is smaller than the total weight of paths in $\cQ=\cP_0$. Let us first bound the total weight of paths deleted by deletions of type 1. In every step when we delete the paths of the form $x_0x_1\cdots x_{2i-2}x'_{2i-1}x_{2i}\cdots x_{2k}$, we delete at most $2\alpha$ paths and each has weight $1/\prod_{h=1}^k d(x_{2h-2},x_{2h})$. Any pair consisting of $i$ and $(x_0,x_1,\ldots,x_{2i-2},x_{2i},\ldots, x_{2k})$ features at most once, so the total weight of paths that get deleted by type 1 deletions is at most
\begin{align*}
    \sum_{i, x_0,x_1\ldots,x_{2i-2},x_{2i},\ldots, x_{2k}} 2\alpha/\prod_{h=1}^k d(x_{2h-2},x_{2h})
    &\leq \sum_{i,x_0,x_2,\ldots,x_{2k}} \prod_{h\neq i} d(x_{2h-2},x_{2h})\cdot 2\alpha/\prod_{h=1}^k d(x_{2h-2},x_{2h}) \\
    &= \sum_{i,x_0,x_2,\ldots,x_{2k}} 2\alpha/d(x_{2i-2},x_{2i}) 
    \leq k\cdot (Kd)^{k+1} 2\alpha/C.
\end{align*}
By the above claim and since $C$ is sufficiently large compared to $k$, $K$, $\alpha$ and $L$, this is less than a third of the total weight in $\cQ$.

Let us now bound the total weight of paths removed by type 2 deletions. In every step when we delete the paths of the form $x_0x_1\cdots x_{2i}x'_{2i+1}x'_{2i+2}x_{2i+3}\cdots x_{2k}$, the total weight removed is
\begin{align*}
    &\sum_{x'_{2i+1},x'_{2i+2}} \left(d(x_{2i},x'_{2i+2})d(x'_{2i+2},x_{2i+4})\prod_{h\neq i,i+1} d(x_{2h},x_{2h+2})\right)^{-1} \\
    &\leq \sum_{x'_{2i+2}} \left(d(x'_{2i+2},x_{2i+4})\prod_{h\neq i,i+1} d(x_{2h},x_{2h+2})\right)^{-1} \\
    &\leq \sum_{x'_{2i+2}} \left(C\prod_{h\neq i,i+1} d(x_{2h},x_{2h+2})\right)^{-1} \leq 2\alpha \left(C\prod_{h\neq i,i+1} d(x_{2h},x_{2h+2})\right)^{-1},
\end{align*}
where the last inequality holds since by the definition of the deletion process there are at most $2\alpha$ choices for $x'_{2i+2}$.
Hence, the total weight removed in all type 2 deletions combined is at most
\begin{align*}
    &\sum_{i,x_0,x_1,\ldots,x_{2i},x_{2i+3},\dots,x_{2k}} 2\alpha \left(C\prod_{h\neq i,i+1} d(x_{2h},x_{2h+2})\right)^{-1} \\
    &\leq \sum_{i,x_0,x_2,x_4,\dots,x_{2i},x_{2i+3},x_{2i+4},x_{2i+6},x_{2i+8},\dots,x_{2k}} 2\alpha/C \leq \sum_i (Kd)^{k+1} 2\alpha/C \leq k(Kd)^{k+1} 2\alpha/C,
\end{align*}
where the penultimate inequality follows from the fact that $x_0,x_2,\ldots,x_{2i},x_{2i+4},x_{2i+6},\ldots,x_{2k}$ are neighbours of $v$ and $x_{2i+3}$ is a neighbour of $x_{2i+4}$. Hence, the total weight removed by type 2 deletions is less than a third of the total weight in $\cQ$. An almost identical argument shows that the total weight removed by type 3 deletions is also less than a third of the total weight in $\cQ$, so we can conclude that the total weight in $\cP$ is positive, therefore $\cP$ is nonempty.
\end{proof}

\section{The prisms: Proof of Theorem~\ref{thm:Main1} }\label{sec:prisms}
\subsection{A short proof of $\ex(n,C_{2\ell}^{\square})=O(n^{3/2})$ for $\ell\ge 7$}\label{shorter proof}

Since the proof of Theorem \ref{thm:Main1} is quite long, we first give a short proof of the slightly weaker statement that $\ex(n,C_{2\ell}^{\square})=O(n^{3/2})$ for $\ell\ge 7$.

By Lemma~\ref{lem:JiangSeiverK-almost}, we can assume that the host graph $G$ is $K$-almost regular for some absolute constant $K$ and has average degree $d=cn^{1/2}$, where $c$ is at least a large constant. Moreover, using that every graph has a bipartite subgraph containing at least half of its edges and Lemma \ref{lem:SubgraphLargeDeg}, we may assume that in addition $G$ is bipartite. Let $T$ be the constant $C_0$ from Lemma \ref{lem:large codegree rare}. Call a copy of $4$-cycle $xyzw$ \emph{thin} if both of the codegrees $d(x,z)$ and $d(y,w)$ are at most $Td^{1/2}$ and call it \emph{thick} otherwise.
Lemma~\ref{Supersaturation of cycles} shows that $G$ contains at least $c_{1}d^{4}$ copies of $C_{4}$, where $c_{1}$ is an absolute constant.

Suppose first that at least $\frac{c_{1}d^{4}}{2}$ many $4$-cycles in $G$ are thin. We then define an auxiliary graph $\mathcal{G}$ as follows. Let $V(\mathcal{G}):=\{(x,y):x,y\in V(G),xy\in E(G)\}$ and let two vertices $\boldsymbol{s}=(s_{1},s_{2})$ and $\boldsymbol{t}=(t_{1},t_{2})$ be adjacent in $\mathcal{G}$ if $s_{1}s_{2}t_{1}t_{2}s_{1}$ is a thin $C_{4}$ in $G$. Note that if there exists a $2\ell$-cycle $(\boldsymbol{w}_{1},\boldsymbol{w}_{2},\ldots,\boldsymbol{w}_{2\ell})$ with $\boldsymbol{w}_{i}=(w_{i}^{(1)},w_{i}^{(2)})$ in this auxiliary graph $\mathcal{G}$ such that $\boldsymbol{w}_{i}\cap\boldsymbol{w}_{j}=\emptyset$ for any $i\neq j$, then there is a copy of $C_{2\ell}^{\square}$ in $G$. 

By the definition of the auxiliary graph $\mathcal{G}$, we can see that $v(\mathcal{G})=m:=2e(G)=nd$ and $e(\mathcal{G})\geq 4\cdot\frac{c_{1}d^{4}}{2}=2c_{1}c^{8/3}m^{4/3}$. By Lemma~\ref{lem:SubgraphLargeDeg}, there is a subgraph $\mathcal{G}'$ with $e(\mathcal{G}')\ge c_{1}c^{8/3}m^{4/3}$ and $\delta(\mathcal{G}')\ge c_{1}c^{8/3}m^{1/3}$.

\begin{claim}\label{claim:conflict777}
    For any vertex $u\in V(G)$ and any $\boldsymbol{s}=(s_{1},s_{2})\in V(\mathcal{G}')$, the vertex $\boldsymbol{s}$ is adjacent to at most $Cd^{-1/2}d_{\mathcal{G}'}(\boldsymbol{s})$ vertices in $\mathcal{G}'$ which contain element $u$, where $C=\frac{2T}{c_{1}c^{2}}$.
\end{claim}
\begin{poc}
Since thin $4$-cycles $xyzw$ in $G$ satisfy $d(x,z)\leqslant Td^{1/2}$ and $d(y,w)\leqslant Td^{1/2}$, and $d_{\mathcal{G}'}(\mathbf{s})\ge c_{1}c^{8/3}m^{1/3}=c_1c^3n^{1/2}$, we see that each vertex $\mathbf{s}$ in $\mathcal{G}'$ has at most $2Td^{1/2}\le Cd^{-1/2}d_{\mathcal{G}'}(\mathbf{s})$ neighbors which contain some fixed element $u$.
\end{poc}

As $Cd^{-1/2}<(2^{20}\ell^{3}(\log{m})^{4}m^{1/\ell})^{-1}$ holds when $\ell\ge 7$, by Lemma~\ref{lem:UpperBoundHomC2k} (with $\alpha=Cd^{-1/2}$ and with $\boldsymbol{s}\sim \boldsymbol{t}$ if and only if $\boldsymbol{s}\cap \boldsymbol{t}\neq \emptyset$), we obtain a copy of $C_{2\ell}$ in $\mathcal{G}$ in which all pairs of vertices are disjoint. This corresponds to a copy of
$C_{2\ell}^{\square}$ in $G$.

It remains to consider the case that there are at least $\frac{c_{1}d^{4}}{2}$ many thick $4$-cycles. Without loss of generality, we can assume that there are at least $\frac{c_{1}d^{4}}{4}$ many thick $4$-cycles $xyzw$ satisfying $d(y,w)\geqslant Td^{1/2}$. By the pigeonhole principle, there is some edge, say $xy\in E(G)$, which can be extended to at least $\frac{c_{1}d^{4}}{4e(G)}=\frac{c_{1}c^{3}n^{1/2}}{8}$ such thick $4$-cycles. Then by Lemma~\ref{lem:large codegree rare}, $G$ contains a copy of $C_{2\ell}^{\square}$.

\begin{rmk}
Setting $d=Cn^{8/13}(\log{n})^{24/13}$ in the above argument with a sufficiently large constant $C$ shows that $\textup{ex}(n,C_{6}^{\square})=O(n^{21/13}(\log{n})^{24/13})$.
\end{rmk}

In the next few sections, we provide the proof of Theorem~\ref{thm:Main1}.

\subsection{Building nice copies of $P_{\ell+1}^{\square}$}
In this section, we fix an integer $\ell\geq 2$.
A key ingredient in the proof of Theorem~\ref{thm:Main1} is a weighted count of certain copies of the Cartesian product $P_{\ell+1}^{\square}=P_{\ell+1}\square K_2$.
The following definition will be used in the proof.

\begin{defn}
    Let us call an $n$-vertex graph $H$ with average degree $d$ \emph{clean} if for any $uv\in E(H)$, $u$ has at least $d/16$ neighbours $w$ in $H$ such that $d_{H}(v,w)\geq \frac{d^2}{128n}$. 
\end{defn}

We will use the following lemma from~\cite{2022TuranGrid}.

\begin{lemma}[\cite{2022TuranGrid}] \label{lem:cleaned subgraph}
    Let $G$ be an $n$-vertex graph with $\alpha n^{3/2}$ edges. Then $G$ has a subgraph $H$ on the same vertex set such that $e(H)\geq \frac{1}{2}\alpha n^{3/2}$ and for any $uv\in E(H)$, $u$ has at least $\frac{1}{8}\alpha n^{1/2}$ neighbours $w$ in $H$ with $d_{H}(v,w)\geq \alpha^2/32$.
\end{lemma}

Then Lemma \ref{lem:cleaned subgraph} states that any graph with average degree at least $2d$ contains a clean subgraph with average degree at least $d$.

Let $G$ be a graph and let distinct vertices $x_i,y_i$ for $0\leq i\leq \ell$ form a copy of $P_{\ell+1}^{\square}$, where $x_iy_i\in E(G)$ for every $i$ and $x_{i-1}x_i,y_{i-1}y_i\in E(G)$ for every $1\leq i\leq \ell$ (see Figure~\ref{figure:K_2 path}). Now the \emph{weight} of this copy is defined to be $1/\prod_{i=1}^{\ell} \max(d_G(x_{i-1},y_i),\frac{d^2}{n})$.

\begin{lemma} \label{lem:lower bound all}
    There exists some $C=C(\ell)$ such that the following holds. Let $G$ be an $n$-vertex clean graph with average degree $d\geq Cn^{1/2}$. Then the total weight of copies of $P_{\ell+1}^{\square}$ in $G$ is at least $\Omega_\ell(nd^{\ell+1})$.
\end{lemma}

\begin{proof}
Choose $C$ large enough so that $\frac{d^2}{128n}\geq 10\ell$. Now we can build copies of $P_{\ell+1}^{\square}$ as follows. Choose $x_{0}y_{0}$ to be an arbitrary edge of $G$. Since $G$ is clean, $y_0$ has at least $d/16$ neighbours $y_1$ in $G$ such that $d_G(x_0,y_1)\geq \frac{d^2}{128n}$. Given such a choice (and since $d_G(x_0,y_1)\geq 10\ell$), there are at least $d_G(x_0,y_1)/2$ ways to choose $x_1$ to be a common vertex of $x_0$ and $y_1$ which is distinct from all the previous vertices. More generally, given $x_i,y_i$ for all $i\leq t$, there are at least $d/16$ ways to choose $y_{t+1}$ to be a neighbour of $y_t$ in $G$ such that $d_G(x_t,y_{t+1})\geq \frac{d^2}{128n}\geq 10\ell$, at least half of these choices are distinct from all previous vertices, and for each such choice there are at least $d_G(x_t,y_{t+1})/2$ ways to choose $x_{t+1}$. Overall, we find that the total weight of the copies of $P_{\ell+1}^{\square}$ constructed by this procedure is at least
\begin{equation*}
    nd\cdot \bigg(\frac{d}{32}\bigg)^{\ell}\cdot\prod_{i=1}^{\ell}\frac{d(x_{i-1},y_{i})}{2}\cdot \frac{1}{\prod_{i=1}^{\ell}\max(d(x_{i-1},y_{i}),\frac{d^{2}}{n})}=\Omega_{\ell}(nd^{\ell+1}).
\end{equation*}
\end{proof}

\begin{defn}
    For distinct vertices $w,z,w',z'$, let us call the $4$-tuple $(w,z,w',z')$ \emph{rich} if $wz,w'z'\in E(G)$, and moreover there are at least $4\ell$ pairwise vertex-disjoint edges $xy\in E(G)$ such that $wx,xw',zy,yz'\in E(G)$.
\end{defn}

The next two lemmas show that in $C_{2\ell}^{\square}$-free graphs, rich $4$-tuples are ``rare".

\begin{lemma} \label{lem:rich rare}
    Let $G$ be a bipartite graph and let $uv\in E(G)$. If $G$ has at least $\ell$ pairwise vertex-disjoint edges $w_iz_i$ ($1\leq i\leq \ell$) such that $w_1,w_2,\dots,w_{\ell}\in N(v)\setminus \{u\}$, $z_1,z_2,\dots,z_{\ell}\in N(u)\setminus\{v\}$ and $(w_{i-1},z_{i-1},w_i,z_i)$ is rich for every $2\leq i\leq \ell$, then $G$ contains a copy of $C_{2\ell}^{\square}$.
\end{lemma}

\begin{proof}
    It suffices to find for each $2\leq i\leq \ell$ an edge $x_iy_i$ such that $w_{i-1}x_i,x_iw_i,z_{i-1}y_i,y_iz_i\in E(G)$ and the vertices $x_2,\dots,x_{\ell},y_2,\dots,y_{\ell}$ are distinct from each other and from the vertices $u,v,w_1,\dots,w_{\ell},z_1,\dots,z_{\ell}$. Indeed, if such vertices exist, then $u$, $v$, $w_i$, $z_i$, $x_i$ and $y_i$ together form a copy of $C_{2\ell}^{\square}$; see Figure~\ref{figure: rich}. Now notice that suitable vertices $x_i,y_i$ can be found greedily for each $i$, using the condition that $(w_{i-1},z_{i-1},w_i,z_i)$ is rich.
\end{proof}

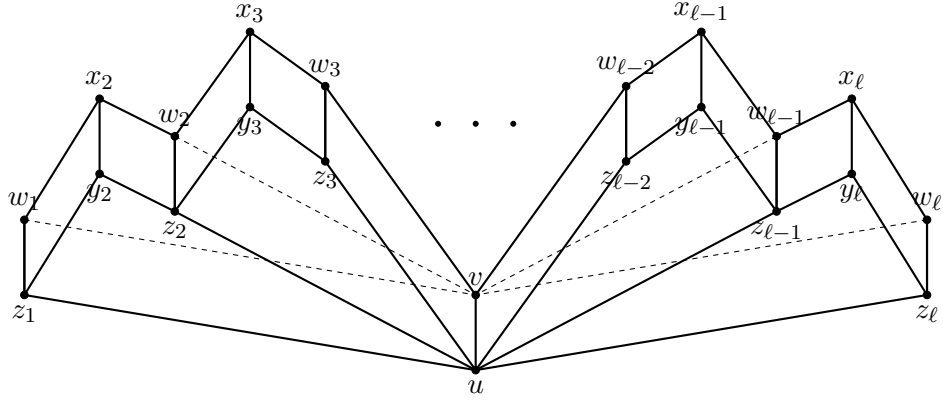
\begin{figure}
\centering
\begin{tikzpicture}[scale=1]

\node[shape = circle,draw = black,fill,inner sep=0pt,minimum size=1.0mm] at (6,-1) {};\node[above] at (6,-1) {$v$};

\node[shape = circle,draw = black,fill,inner sep=0pt,minimum size=1.0mm] at (6,-2) {};\node[below] at (6,-2) {$u$};
\draw[thick, fill opacity=0.3] (6,-1) -- (6,-2);

\foreach \m in {0,2,4,8,10,12}{
\foreach \n in {1}{
    \node[shape = circle,draw = black,fill,inner sep=0pt,minimum size=1.0mm] at (\m,{2-(\m-6)^2/18-\n})  {}; 
    \draw[thick, fill opacity=0.3] (\m,{2-(\m-6)^2/18-\n}) -- (6,{-\n-1});
}
\draw[thick, fill opacity=0.3] (\m,{2-(\m-6)^2/18}) -- (\m,{1-(\m-6)^2/18});
}
\foreach \m in {4,8}{
\foreach \n in {0}{
    \node[shape = circle,draw = black,fill,inner sep=0pt,minimum size=1.0mm] at (\m,{2-(\m-6)^2/18-\n})  {}; 
    \draw[thick, fill opacity=0.3] (\m,{2-(\m-6)^2/18-\n}) -- (6,{-\n-1});
}
\draw[thick, fill opacity=0.3] (\m,{2-(\m-6)^2/18}) -- (\m,{1-(\m-6)^2/18});
}
\foreach \m in {0,2,10,12}{
\foreach \n in {0}{
    \node[shape = circle,draw = black,fill,inner sep=0pt,minimum size=1.0mm] at (\m,{2-(\m-6)^2/18-\n})  {}; 
   \draw[dash pattern=on 2pt off 2pt on 2pt off 2pt, fill opacity=0.3] (\m,{2-(\m-6)^2/18-\n}) -- (6,{-\n-1});
}
\draw[thick, fill opacity=0.3] (\m,{2-(\m-6)^2/18}) -- (\m,{1-(\m-6)^2/18});
}

\foreach \m in {1,2,3}{
\node[above] at ({2*\m-2},{2-(2*\m-8)^2/18}) {$w_{\m}$};
\node[below] at ({2*\m-2},{1-(2*\m-8)^2/18}) {$z_{\m}$};
}
\foreach \m in {12}{
\node[above] at (\m,{2-(\m-6)^2/18}) {$w_{\ell}$};
\node[below] at (\m,{1-(\m-6)^2/18}) {$z_{\ell}$};
}
\foreach \m in {10}{
\node[above] at (\m,{2-(\m-6)^2/18}) {$w_{\ell-1}$};
\node[below] at (\m,{1-(\m-6)^2/18}) {$z_{\ell-1}$};
}
\foreach \m in {8}{
\node[above] at (\m,{2-(\m-6)^2/18}) {$w_{\ell-2}$};
\node[below] at (\m,{1-(\m-6)^2/18}) {$z_{\ell-2}$};
}

\foreach \m in {5.5,6,6.5}{
 \node[shape = circle,draw = black,fill,inner sep=0pt,minimum size=0.65mm] at (\m,{1.5-2/9})  {};
}

\foreach \m in {1,3,9,11}{
\foreach \n in {0,1}{
    \node[shape = circle,draw = black,fill,inner sep=0pt,minimum size=1.0mm] at (\m,{3-(\m-6)^2/18-\n})  {};
    \draw[thick, fill opacity=0.3] (\m,{3-(\m-6)^2/18-\n}) -- (\m+1,{2-(\m-5)^2/18-\n});
    \draw[thick, fill opacity=0.3] (\m,{3-(\m-6)^2/18-\n}) -- (\m-1,{2-(\m-7)^2/18-\n});
}
\draw[thick, fill opacity=0.3] (\m,{3-(\m-6)^2/18}) -- (\m,{2-(\m-6)^2/18});
}
\foreach \m in {1}{
\node[above] at (\m,{3-(\m-6)^2/18}) {$x_{2}$};
\node[below] at (\m,{2-(\m-6)^2/18}) {$y_{2}$};
}
\foreach \m in {3}{
\node[above] at (\m,{3-(\m-6)^2/18}) {$x_{3}$};
\node[below] at (\m,{2-(\m-6)^2/18}) {$y_{3}$};
}
\foreach \m in {9}{
\node[above] at (\m,{3-(\m-6)^2/18}) {$x_{\ell-1}$};
\node[below] at (\m,{2-(\m-6)^2/18}) {$y_{\ell-1}$};
}
\foreach \m in {11}{
\node[above] at (\m,{3-(\m-6)^2/18}) {$x_{\ell}$};
\node[below] at (\m,{2-(\m-6)^2/18}) {$y_{\ell}$};
}

\end{tikzpicture}
\caption{\label{figure: rich} Finding a $C^{\square}_{2\ell}$ from rich $4$-tuples in Lemma \ref{lem:rich rare}}
\end{figure}

From now on, let $K$ be some fixed absolute constant and let $C_0$ be the constant provided by Lemma \ref{lem:large codegree rare}.

\begin{lemma} \label{lem:rich with small codeg}
    Let $G$ be a bipartite $n$-vertex $C_{2\ell}^{\square}$-free graph with average degree $d\geq C_0n^{1/2}$ and maximum degree at most $Kd$, and let $uv\in E(G)$. Then the number of rich $4$-tuples $(w,z,w',z')$ such that $uz,uz',vw,vw'\in E(G)$ and $d(u,w),d(u,w'),d(v,z),d(v,z')\leq C_0d^{1/2}$ is $O_{K,\ell}(d^2)$.
\end{lemma}

\begin{proof}
Define an auxiliary graph $H$ whose vertex set consists of those edges $wz$ in $G$ for which $uvwz$ is a cycle with $d(u,w),d(v,z)\leq C_0 d^{1/2}$, and in which there is an edge between $wz$ and $w'z'$ if $(w,z,w',z')$ is rich. The lemma then reduces to showing that $H$ has $O_{K,\ell}(d^2)$ edges. Since the number of vertices in $H$ is at most $Kd\cdot C_0d^{1/2}$, it suffices to prove that $H$ has average degree $O_{K,\ell}(d^{1/2})$. Assume, for the sake of contradiction, that the average degree of $H$ is at least $100\ell C_0 d^{1/2}$. Then $H$ has a non-empty subgraph $H'$ in which the minimum degree is at least $50\ell C_0 d^{1/2}$. Hence, we may find vertices $w_iz_i$ in $H'$ for $1\leq i\leq \ell$ such that for each $2\leq i\leq \ell$, the vertices $w_{i-1}z_{i-1}$ and $w_iz_i$ are adjacent in $H'$, and $w_1,\dots,w_\ell,z_1,\dots,z_{\ell}$ are distinct. Indeed, given $w_1z_1,\dots,w_tz_t$, we just need to choose a neighbour $w_{t+1}z_{t+1}$ of $w_tz_t$ in $H'$ which does not have a vertex in common with any of $w_1z_1,\dots,w_tz_t$. By the codegree condition, at most $4\ell C_0d^{1/2}$ vertices of $H'$ are forbidden, so the minimum degree condition of $H'$ guarantees that we can find a suitable $w_{t+1}z_{t+1}$.

However, the vertices $w_1,\dots,w_\ell,z_1,\dots,z_\ell$ now satisfy the conditions of Lemma \ref{lem:rich rare}, so that lemma implies that $G$ contains a copy of $C_{2\ell}^{\square}$, which is a contradiction.
\end{proof}

\begin{lemma} \label{lem:upper bound bad}
    Let $G$ be a bipartite $n$-vertex $C_{2\ell}^{\square}$-free graph with average degree $d\geq C_0n^{1/2}$ and maximum degree at most $Kd$. Then
    \begin{enumerate}[label=(\alph*)]
        \item the total weight of copies of $P_{\ell+1}^{\square}$ with $d(x_{i-1},y_i)>C_0d^{1/2}$ for some $1\leq i\leq \ell$ is $O_{K,\ell}(n^2 d^{\ell-1})$,
        \item the total weight of copies of $P_{\ell+1}^{\square}$ with $d(x_{i},y_{i-1})>C_0d^{1/2}$ for some $1\leq i\leq \ell$ is $O_{K,\ell}(n^2 d^{\ell-1})$ and
        \item the total weight of copies of $P_{\ell+1}^{\square}$ such that $d(x_{i-1},y_i),d(x_{i},y_{i-1})\leq C_{0}d^{1/2}$ for every $1\leq i\leq \ell$ and $(x_{j-1},y_{j-1},x_{j+1},y_{j+1})$ is rich for some $1\leq j\leq \ell-1$ is $O_{K,\ell}(n^3 d^{\ell-3})$.
    \end{enumerate}
\end{lemma}

\begin{proof}
(a) Fix some $1\leq i\leq \ell$. The number of ways to choose $x_{i-1}$ and $y_{i-1}$ is at most $nd$. Then, by Lemma~\ref{lem:large codegree rare}, the number of possible choices for $x_i$ and $y_i$ (so that they form a 4-cycle with $x_{i-1}$ and $y_{i-1}$ and $d(x_{i-1},y_i)>C_0d^{1/2}$) is at most $C_0d$. If we assign weight $1/\max(d(x_{i-1},y_i),\frac{d^2}{n})$ to each such 4-cycle $x_{i-1}y_{i-1}y_ix_i$, then the total weight of possible 4-cycles is at most $\frac{nd(C_0d)}{d^2/n}\leq O_{K,\ell}(n^2)$. To extend such a $4$-cycle to a copy of $P_{\ell+1}^{\square}$, we can choose $y_{i+1}$ in at most $Kd$ ways, and given such a choice there are at most $d(x_i,y_{i+1})$ possibilities for $x_{i+1}$. Similarly, there are at most $Kd$ choices for $x_{i-2}$ and given such a choice there are at most $d(x_{i-2},y_{i-1})$ possibilities for $y_{i-2}$. Since the weight of a copy of $P_{\ell+1}^{\square}$ is $1/\prod_{i=1}^{\ell} \max(d(x_{i-1},y_i),\frac{d^2}{n})$, it follows fairly easily that the total weight of the desired copies of $P_{\ell+1}^{\square}$ is $O_{K,\ell}(n^2d^{\ell-1})$.

The proof of (b) is almost identical.

(c) Fix some $1\leq j\leq \ell-1$. There are at most $nd$ ways to choose $x_j$ and $y_j$. By Lemma \ref{lem:rich with small codeg}, there are only $O_{K,\ell}(d^2)$ possibilities for the vertices $x_{j-1},y_{j-1},x_{j+1},y_{j+1}$. Then the total weight of all possible copies of $P_{3}^{\square}$ formed by $x_{j-1},y_{j-1},x_j,y_j,x_{j+1},y_{j+1}$ is at most $O_{K,\ell}(nd\cdot d^2/ (d^2/n)^2)$, where the weight of such a subgraph is defined to be $1/\left(\max(d(x_{j-1},y_j),\frac{d^2}{n})\cdot\max(d(x_{j},y_{j+1}),\frac{d^2}{n})\right)$. It follows similarly to case (a) that the total weight of the desired copies of $P_{\ell+1}^{\square}$ is $O_{K,\ell}(n^3 d^{\ell-3})$.
\end{proof}

We now define those copies of $P_{\ell+1}^{\square}$ which we will use to build our $C_{2\ell}^{\square}$. Here and below, it is implicitly assumed that the host graph has average degree $d$.

\begin{defn}
    We say that vertices $x_i,y_i$ (for $0\leq i\leq \ell$) form a \emph{nice} copy of $P_{\ell+1}^{\square}$ if they form a copy of $P_{\ell+1}^{\square}$, for every $1\leq i\leq \ell$ the codegrees satisfy $d(x_{i-1},y_i),d(x_i,y_{i-1})\leq C_0d^{1/2}$, and for every $2\leq i\leq \ell$, the $4$-tuple $(x_{i-2},y_{i-2},x_i,y_i)$ is not rich.    
\end{defn}

We now prove the following lower bound for the total weight of nice copies of $P_{\ell+1}^{\square}$, which is the main result of this subsection.

\begin{lemma} \label{lem:many nice paths}
    If $C$ is sufficiently large compared to $K$ and $\ell$, then the following holds. Let $G$ be a clean, bipartite, $n$-vertex graph with average degree $d\geq Cn^{1/2}$ and maximum degree at most $Kd$. Then the total weight of nice copies of $P_{\ell+1}^{\square}$ in $G$ is $\Omega_\ell(nd^{\ell+1})$.
\end{lemma}

\begin{proof}
By Lemma \ref{lem:lower bound all}, the total weight of copies of $P_{\ell+1}^{\square}$ in $G$ is at least $\Omega_\ell(nd^{\ell+1})$. On the other hand, by Lemma \ref{lem:upper bound bad}, the total weight of the copies that fail to be nice is $O_{K,\ell}(n^2d^{\ell-1}+n^3d^{\ell-3})$. It follows that if $C$ is sufficiently large, then at most half of the total weight is on copies which are not nice, completing the proof of the lemma.
\end{proof}

\subsection{Completing the proof of Theorem \ref{thm:Main1}}

We now define those homomorphic copies of $C_{2\ell}^{\square}$ that are obtained by gluing together two (suitable) nice copies of $P_{\ell+1}^{\square}$.

\begin{defn}
    We say that vertices $x_i,y_i,x_i',y_i'$ (for $0\leq i\leq \ell$) form a nice homomorphic copy of $C_{2\ell}^{\square}$ if $x_0=x_0',y_0=y_0',x_{\ell}=x_{\ell}',y_{\ell}=y_{\ell}'$, both $\{x_i,y_i:0\leq i\leq \ell\}$ and $\{x_i',y_i':0\leq i\leq \ell\}$ form a nice copy of $P_{\ell+1}^{\square}$, each $x_i$ is distinct from all other vertices except possibly $x_i'$ and each $y_i$ is distinct from all other vertices except possibly $y_i'$. The weight of such a homomorphic copy of $C_{2\ell}^{\square}$ is defined to be $\left(\prod_{i=1}^{\ell} \max(d(x_{i-1},y_i),\frac{d^2}{n}) \cdot \prod_{i=1}^{\ell} \max(d(x_{i-1}',y_i'),\frac{d^2}{n})\right)^{-1}$.
\end{defn}

\begin{figure}
\centering
\begin{tikzpicture}[scale=1]

\foreach \m in {1,2.5,3.5,5}{
\foreach \n in {1,2,-1,-2}{
    \node[shape = circle,draw = black,fill,inner sep=0pt,minimum size=1.0mm] at (\m,\n) {};
}
\draw[thick, fill=blue, fill opacity=0.3] (\m,1) -- (\m,2);
\draw[thick, fill=blue, fill opacity=0.3] (\m,-1) -- (\m,-2);
}

\node[shape = circle,draw = black,fill,inner sep=0pt,minimum size=1.0mm] at (-0.5,0.5) {};\node[above] at (-1,0.5) {$x_0 = x'_{0}$};
\node[shape = circle,draw = black,fill,inner sep=0pt,minimum size=1.0mm] at (-0.5,-0.5) {};\node[below] at (-1,-0.5) {$y_0 = y'_{0}$};
\draw[thick, fill=blue, fill opacity=0.3] (-0.5,0.5) -- (-0.5,-0.5); \draw[thick, fill=blue, fill opacity=0.3] (-0.5,0.5) -- (1,2);
\draw[thick, fill=blue, fill opacity=0.3] (-0.5,0.5) -- (1,-1);
\draw[thick, fill=blue, fill opacity=0.3] (-0.5,-0.5) -- (1,-2);
\draw[thick, fill=blue, fill opacity=0.3] (-0.5,-0.5) -- (1,1);

\node[shape = circle,draw = black,fill,inner sep=0pt,minimum size=1.0mm] at (6.5,0.5) {};\node[above] at (7,0.5) {$x_\ell = x'_{\ell}$};
\node[shape = circle,draw = black,fill,inner sep=0pt,minimum size=1.0mm] at (6.5,-0.5) {};\node[below] at (7,-0.5) {$y_\ell = y'_{\ell}$};
\draw[thick, fill=blue, fill opacity=0.3] (6.5,0.5) -- (6.5,-0.5); \draw[thick, fill=blue, fill opacity=0.3] (6.5,0.5) -- (5,2);
\draw[thick, fill=blue, fill opacity=0.3] (6.5,0.5) -- (5,-1);
\draw[thick, fill=blue, fill opacity=0.3] (6.5,-0.5) -- (5,-2);
\draw[thick, fill=blue, fill opacity=0.3] (6.5,-0.5) -- (5,1);

\node[above] at (1,2) {$x_1$};
\node[above] at (2.5,2) {$x_i$};
\node[above] at (3.5,2) {$x_{i+1}$};
\node[above] at (5,2) {$x_{\ell-1}$};
\node[above] at (1,-1) {$x'_1$};
\node[above] at (2.5,-1) {$x'_i$};
\node[above] at (3.5,-1) {$x'_{i+1}$};
\node[above] at (5,-1) {$x'_{\ell-1}$};

\node[below] at (1,1) {$y_1$};
\node[below] at (2.5,1) {$y_i$};
\node[below] at (3.5,1) {$y_{i+1}$};
\node[below] at (5,1) {$y_{\ell-1}$};
\node[below] at (1,-2) {$y'_1$};
\node[below] at (2.5,-2) {$y'_i$};
\node[below] at (3.5,-2) {$y'_{i+1}$};
\node[below] at (5,-2) {$y'_{\ell-1}$};

\node at (1.75,1.5) {$\cdots$};
\node at (4.25,1.5) {$\cdots$};
\node at (1.75,-1.5) {$\cdots$};
\node at (4.25,-1.5) {$\cdots$};

\foreach \n in {1,2,-1,-2}{
\draw[dash pattern=on 2pt off 3pt on 4pt off 4pt, fill opacity=0.3] (1,\n) -- (2.5,\n);
\draw[thick, fill opacity=0.3] (2.5,\n) -- (3.5,\n);
\draw[dash pattern=on 2pt off 3pt on 4pt off 4pt, fill opacity=0.3](3.5,\n) -- (5,\n);
}
\end{tikzpicture}
\caption{\label{figure: prisms}$C_{2\ell}^{\square}$}
\end{figure}
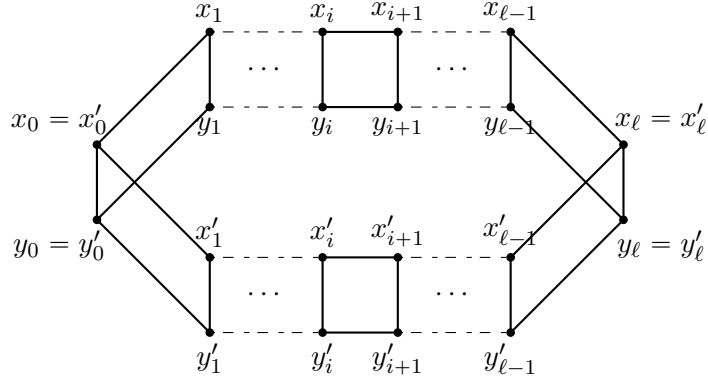

The next lemma shows that there is large total weight on nice homomorphic copies of $C_{2\ell}^{\square}$.

\begin{lemma} \label{lem:lower bound nice}
    If $C$ is sufficiently large compared to $K$ and $\ell$, then the following holds. Let $G$ be a clean, bipartite, $n$-vertex graph with average degree $d\geq Cn^{1/2}$ and maximum degree at most $Kd$. Then the total weight of nice homomorphic copies of $C_{2\ell}^{\square}$ in $G$ is $\Omega_{\ell}(d^{2\ell})$.
\end{lemma}

\begin{proof}
By Lemma \ref{lem:many nice paths}, the total weight of nice copies of $P_{\ell+1}^{\square}$ is $\Omega_\ell(nd^{\ell+1})$. Let $X_0\cup X_1\cup \dots \cup X_{\ell}\cup Y_0\cup Y_1 \cup \dots \cup Y_{\ell}$ be a random partition of $V(G)$, where any $v\in V(G)$ belongs to each of the $2\ell+2$ sets with probability $\frac{1}{2\ell+2}$, independently of the other vertices. Then the expected value of the total weight of those nice copies of $P_{\ell+1}^{\square}$ which satisfy $x_i\in X_i$ and $y_i\in Y_i$ for all $i$ is $\Omega_{\ell}(nd^{\ell+1})$. Choose a partition for which this total weight is $\Omega_{\ell}(nd^{\ell+1})$.  Now observe that whenever both $\{x_i,y_i:0\leq i\leq \ell\}$ and $\{x'_i,y'_i:0\leq i\leq \ell\}$ form a nice copy of $P_{\ell+1}^{\square}$ with $x_i,x'_i\in X_i$, $y_i,y'_i\in Y_i$ and $x_0=x_0',y_0=y_0',x_{\ell}=x_{\ell}',y_{\ell}=y_{\ell}'$, then the vertices $x_i,x'_i,y_i,y'_i$ together form a nice homomorphic copy of $C_{2\ell}^{\square}$. Given edges $x_{0}y_{0}$ and $x_{\ell}y_{\ell}$, write $f(x_0,y_0,x_{\ell},y_{\ell})$ for the total weight of nice copies of $P_{\ell+1}^{\square}$ extending $x_0,y_0,x_{\ell},y_{\ell}$ and satisfying $x_i\in X_i,y_i\in Y_i$ for all $i$. Then the total weight of nice homomorphic copies of $C_{2\ell}^{\square}$ is at least $$\sum_{x_0,y_0,x_{\ell},y_{\ell}} f(x_0,y_0,x_{\ell},y_{\ell})^2\geq (nd)^2 \left(\frac{\sum_{x_0,y_0,x_{\ell},y_{\ell}} f(x_0,y_0,x_{\ell},y_{\ell})}{(nd)^2}\right)^2\geq (nd)^2\frac{\Omega_{\ell}((nd^{\ell+1})^2)}{(nd)^4}\geq \Omega_{\ell}(d^{2\ell}),$$
where the first inequality follows from the convexity of $x^2$ and the second inequality follows from our lower bound on the total weight of nice copies of $P_{\ell}^{\square}$ with $x_i\in X_i$ and $y_i\in Y_i$ for all $i$.
\end{proof}

In what follows we give an upper bound for the total weight of nice homomorphic copies of $C_{2\ell}^{\square}$ in $C_{2\ell}^{\square}$-free graphs. By the $C_{2\ell}^{\square}$-freeness, in any such copy we have $x_i=x'_i$ or $y_i=y'_i$ for some $1\leq i\leq \ell-1$. A key lemma for bounding the total weight of nice homomorphic copies of $C_{2\ell}^{\square}$ is as follows.

\begin{lemma} \label{lem:count nonrich extensions}
    Let $x_{i-1}y_{i-1}$ and $x_{i+1}y_{i+1}$ be disjoint edges and assume that $(x_{i-1},y_{i-1},x_{i+1},y_{i+1})$ is not rich. Then the total weight of copies of $P_{3}^{\square}$ formed by vertices $x_{i-1},y_{i-1},x_i,y_i,x_{i+1},y_{i+1}$ is at most $8\ell$, where the weight of one such copy of $P_{3}^{\square}$ is defined to be $(\max(d(x_{i-1},y_i),d^2/n)\cdot\max(d(x_{i},y_{i+1}),d^2/n))^{-1}$.
\end{lemma}

\begin{proof}
Take a maximal set of disjoint edges $uv$ such that $x_{i-1},y_{i-1},u,v,x_{i+1},y_{i+1}$ form a copy of $P_{3}^{\square}$. Let $S$ be the set of vertices appearing in these edges. By the assumption that $(x_{i-1},y_{i-1},x_{i+1},y_{i+1})$ is not rich, we choose at most $4\ell$ edges, so $|S|\leq 8\ell$. By the maximality assumption, any edge $x_iy_i$ for which $x_{i-1},y_{i-1},x_i,y_i,x_{i+1},y_{i+1}$ form a copy of $P_{3}^{\square}$ satisfies that $x_i\in S$ or $y_i\in S$. For any given $x_i\in S$, the total weight of copies of $P_{3}^{\square}$ extending $x_{i-1},y_{i-1},x_i,x_{i+1},y_{i+1}$ is at most 1, since there are at most $d(x_i,y_{i+1})$ of them and the weight of each of them is at most $1/d(x_i,y_{i+1})$. Similarly, for any given $y_i\in S$, the total weight of copies of $P_{3}^{\square}$ extending $x_{i-1},y_{i-1},y_i,x_{i+1},y_{i+1}$ is at most 1. This completes the proof.
\end{proof}

We will also use the following very simple lemma. Here and below the weight of a $4$-cycle with vertices $x_i,y_i,x_{i+1},y_{i+1}$ is defined to be $1/\max(d(x_i,y_{i+1}),d^2/n)$.

\begin{lemma} \label{lem:count extensions}
    Let $G$ be an $n$-vertex graph with average degree $d$ and maximum degree at most $Kd$.
    Let $x_iy_i$ be an edge in $G$.
    \begin{enumerate}[label=(\roman*)]
        \item The total weight on $4$-cycles with vertices $x_i,y_i,x_{i+1},y_{i+1}$ is at most $Kd$.
        \item The total weight on $4$-cycles with vertices $x_{i-1},y_{i-1},x_{i},y_{i}$ is at most $Kd$.
        \item For any vertex $x_{i+1}$, the number of $4$-cycles with vertices $x_i,y_i,x_{i+1},y_{i+1}$ such that $d(x_{i+1},y_{i})\leq C_0d^{1/2}$ is at most $C_0d^{1/2}$. In particular, the total weight on these cycles is at most $C_0d^{1/2}$.
        \item For any vertex $y_{i+1}$, the number of $4$-cycles with vertices $x_i,y_i,x_{i+1},y_{i+1}$ such that $d(x_{i},y_{i+1})\leq C_0d^{1/2}$ is at most $C_0d^{1/2}$. In particular, the total weight on these cycles is at most $C_0d^{1/2}$.
        \item For any vertex $x_{i-1}$, the number of $4$-cycles with vertices $x_{i-1},y_{i-1},x_{i},y_{i}$ such that $d(x_{i-1},y_{i})\leq C_0d^{1/2}$ is at most $C_0d^{1/2}$. In particular, the total weight on these cycles is at most $C_0d^{1/2}$.
        \item For any vertex $y_{i-1}$, the number of $4$-cycles with vertices $x_{i-1},y_{i-1},x_{i},y_{i}$ such that $d(x_{i},y_{i-1})\leq C_0d^{1/2}$ is at most $C_0d^{1/2}$. In particular, the total weight on these cycles is at most $C_0d^{1/2}$.
    \end{enumerate}
\end{lemma}

\begin{proof}
We only prove (i); the proof of (ii) is almost identical and all the other statements are obvious. The number of ways to choose $y_{i+1}$ to be a neighbour of $y_i$ is at most $Kd$. Given such a choice, there are at most $d(x_i,y_{i+1})$ ways to choose $x_{i+1}$ and each such $4$-cycle has weight at most $1/d(x_i,y_{i+1})$.
\end{proof}

The following lemma deals with almost all types of nice homomorphic copies of $C_{2\ell}^{\square}$.

\begin{lemma} \label{lem:internal vertex}
    Let $\ell\geq 4$ and let $G$ be a $C_{2\ell}^{\square}$-free $n$-vertex graph with average degree $d$ and maximum degree at most $Kd$. Then the total weight of nice homomorphic copies of $C_{2\ell}^{\square}$ with
    \begin{enumerate}[label=(\alph*)]
        \item $x'_i=x_i$ for some $2\leq i\leq \ell-2$, or
        \item $y'_i=y_i$ for some $2\leq i\leq \ell-2$
    \end{enumerate}
    is $O_{K,\ell}(nd^{2\ell-2})$.
\end{lemma}

\begin{proof}
By repeated applications of (i) from Lemma \ref{lem:count extensions}, the total weight of possible copies of $P_{\ell+1}^{\square}$ formed by all the $x_k,y_k$ for $0\leq k\leq \ell$ is $O_K(nd^{\ell+1})$. Since $x'_i=x_i$ or $y'_i=y_i$, this already determines one of $x'_i$ and $y'_i$. The other one can be chosen in at most $Kd$ ways. By Lemma \ref{lem:count extensions} (i), given $x'_0$ and $y'_0$, the total weight of the possible $4$-cycles with vertices $x'_{0},y'_{0},x'_{1},y'_{1}$ is at most $Kd$. Similarly, for any $0\leq k\leq i-3$, given $x'_k$ and $y'_k$, the total weight of the possible $4$-cycles with vertices $x'_{k},y'_{k},x'_{k+1},y'_{k+1}$ is at most $Kd$. By Lemma \ref{lem:count nonrich extensions}, given $x'_{i-2},y'_{i-2},x'_i,y'_i$, the total weight of the possible copies of $P_{3}^{\square}$ formed by $x'_{i-2},y'_{i-2},x'_{i-1},y'_{i-1},x'_i,y'_i$ is at most $8\ell$. Again by Lemma \ref{lem:count extensions} (i), for any $i\leq k\leq \ell-3$, given $x'_k$ and $y'_k$, the total weight of the possible $4$-cycles with vertices $x'_{k},y'_{k},x'_{k+1},y'_{k+1}$ is at most $Kd$. Finally, by Lemma \ref{lem:count nonrich extensions}, given $x'_{\ell-2},y'_{\ell-2},x'_{\ell},y'_{\ell}$, the total weight of the possible copies of $P_{3}^{\square}$ formed by $x'_{\ell-2},y'_{\ell-2},x'_{\ell-1},y'_{\ell-1},x'_{\ell},y'_{\ell}$ is at most $8\ell$. Overall, we obtain an upper bound $O_{K}(nd^{\ell+1}\cdot (Kd)^{\ell-3}\cdot (8\ell)^2)$, which is indeed of the form $O_{K,\ell}(nd^{2\ell-2})$.
\end{proof}

We now bound the total weight of those nice homomorphic copies of $C_{2\ell}^{\square}$ which have at least two ``degeneracies", but are not covered by Lemma \ref{lem:internal vertex}.

\begin{lemma} \label{lem:two degeneracies}
    Let $\ell\geq 4$ and let $G$ be an $n$-vertex graph with average degree $d$ and maximum degree at most $Kd$. Then the total weight of nice homomorphic copies of $C_{2\ell}^{\square}$ with
    \begin{enumerate}[label=(\alph*)]
        \item $x'_1=x_1$ and $x'_{\ell-1}=x_{\ell-1}$, or
        \item $x'_1=x_1$ and $y'_{\ell-1}=y_{\ell-1}$, or
        \item $x'_1=x_1$ and $y'_1=y_1$, or
        \item $x'_{\ell-1}=x_{\ell-1}$ and $y'_{1}=y_{1}$, or
        \item $x'_{\ell-1}=x_{\ell-1}$ and $y'_{\ell-1}=y_{\ell-1}$, or
        \item $y'_1=y_1$ and $y'_{\ell-1}=y_{\ell-1}$
    \end{enumerate}
    is $O_{K,\ell}(nd^{2\ell-2})$.
\end{lemma}

\begin{proof}
All of these statements follow from repeated applications of Lemma \ref{lem:count extensions} and one application of Lemma \ref{lem:count nonrich extensions}. Since they are all quite similar and straightforward, we only prove (a).

By repeated applications of (i) from Lemma \ref{lem:count extensions}, the total weight of possible copies of $P_{\ell+1}^{\square}$ formed by all the $x_k,y_k$ for $0\leq k\leq \ell$ is $O_K(nd^{\ell+1})$.
By Lemma \ref{lem:count extensions} (iii), given $x'_0$, $y'_0$ and the vertex $x'_1=x_1$, the total weight of the possible $4$-cycles with vertices $x'_{0},y'_{0},x'_{1},y'_{1}$ is at most $C_0d^{1/2}$. By Lemma \ref{lem:count extensions} (v), given $x'_\ell$, $y'_\ell$ and the vertex $x'_{\ell-1}=x_{\ell-1}$, the total weight of the possible $4$-cycles with vertices $x'_{\ell-1},y'_{\ell-1},x'_{\ell},y'_{\ell}$ is at most $C_0d^{1/2}$. Moreover, by Lemma \ref{lem:count extensions} (i), for any $1\leq k\leq \ell-4$, given $x'_k$ and $y'_k$, the total weight of the possible $4$-cycles with vertices $x'_{k},y'_{k},x'_{k+1},y'_{k+1}$ is at most $Kd$.
Finally, by Lemma \ref{lem:count nonrich extensions}, given $x'_{\ell-3},y'_{\ell-3},x'_{\ell-1},y'_{\ell-1}$, the total weight of the possible copies of $P_{3}^{\square}$ formed by $x'_{\ell-3},y'_{\ell-3},x'_{\ell-2},y'_{\ell-2},x'_{\ell-1},y'_{\ell-1}$ is at most $8\ell$. Overall, we obtain an upper bound $O_{K}(nd^{\ell+1}\cdot (C_0d^{1/2})^2\cdot (Kd)^{\ell-4}\cdot 8\ell)$, which is indeed of the form $O_{K,\ell}(nd^{2\ell-2})$.
\end{proof}

It remains to bound the total weight of those nice homomorphic copies of $C_{2\ell}^{\square}$ which have a single degeneracy, at $x'_1$, $y'_1$, $x'_{\ell-1}$ or $y'_{\ell-1}$.

\begin{lemma} \label{lem:one degeneracy}
    Let $\ell\geq 4$ and let $G$ be a $C_{2\ell}^{\square}$-free $n$-vertex graph with average degree $d$ and maximum degree at most $Kd$. Then the total weight of nice homomorphic copies of $C_{2\ell}^{\square}$ in which for all $1\leq i\leq \ell-1$, we have $x'_i\neq x_i$ and $y'_i\neq y_i$ except we have precisely one of
    \begin{enumerate}[label=(\alph*)]
        \item $x'_1=x_1$, or
        \item $y'_{1}=y_{1}$, or
        \item $x'_{\ell-1}=x_{\ell-1}$, or
        \item $y'_{\ell-1}=y_{\ell-1}$
    \end{enumerate}
    is $O_{K,\ell}(nd^{2\ell-2})$.
\end{lemma}

\begin{proof}
We only prove that the total weight of nice homomorphic copies of $C_{2\ell}^{\square}$ with $x'_1=x_1$ but $x'_i\neq x_i$ for all $2\leq i\leq \ell-1$ and $y'_i\neq y_i$ for all $1\leq i\leq \ell-1$ is at most $O_{K,\ell}(nd^{2\ell-2})$. (The upper bounds for the cases (b), (c) and (d) can be proven almost identically.)

Consistently with the convention in this section, the weight of a subgraph of a nice homomorphic copy of $C_{2\ell}^{\square}$ is the product of the weights of the individual 4-cycles that it contains. With this definition and using Lemma \ref{lem:count extensions} (i), the total weight of the possible copies of $P_{\ell}^{\square}$ with vertices $x_0,y_0,x_1,y_1,\dots,x_{\ell-1},y_{\ell-1}$ is at most $O_K(nd^{\ell})$. By the definition of nice copies of $C_{2\ell}^{\square}$ and Lemma \ref{lem:count extensions} (iii), given $x'_0$, $y'_0$ and the vertex $x'_1=x_1$, the total weight of the possible $4$-cycles with vertices $x'_{0},y'_{0},x'_{1},y'_{1}$ is at most $C_0d^{1/2}$. Hence, the total weight of the possible subgraphs induced by the vertices $x_0,y_0,\dots,x_{\ell-1},y_{\ell-1},x'_1,y'_1$ is $O_{K,\ell}(nd^{\ell+1/2})$. It is therefore sufficient to prove that the total weight of the possible subgraphs spanned by $x_{\ell-1},y_{\ell-1},x_{\ell},y_{\ell},x'_1,y'_1,x'_2,y'_2,\dots,x'_{\ell-1},y'_{\ell-1}$ is $O_{K,\ell}(d^{\ell-5/2})$ for any fixed choices of $x_{\ell-1},y_{\ell-1},x'_1,y'_1$. Observe that these subgraphs are isomorphic to $P_{\ell+1}^{\square}$ (with a non-standard labelling), so by the $C_{2\ell}^{\square}$-freeness assumption, any two such subgraphs which correspond to the same $x_{\ell-1},y_{\ell-1},x'_1,y'_1$ must have another vertex in common (in addition to $x_{\ell-1},y_{\ell-1},x'_1,y'_1$). This implies that for a given $x_{\ell-1},y_{\ell-1},x'_1,y'_1$, all possible sets $\{x_{\ell},y_{\ell},x'_2,y'_2,\dots,x'_{\ell-1},y'_{\ell-1}\}$ intersect a fixed set of size at most $2\ell$. However, as $\ell\geq 4$, it is easy to show by $\ell-3$ applications of Lemma \ref{lem:count extensions} (i) and (ii), one application of Lemma \ref{lem:count extensions} (iii), (iv), (v) or (vi) and one application of Lemma \ref{lem:count nonrich extensions} that for given $x_{\ell-1},y_{\ell-1},x'_1,y'_1$ and given one more vertex from the set $\{x_{\ell},y_{\ell},x'_2,y'_2,\dots,x'_{\ell-1},y'_{\ell-1}\}$, the total weight of all possible subgraphs spanned by $x_{\ell-1},y_{\ell-1},x_{\ell},y_{\ell},x'_1,y'_1,x'_2,y'_2,\dots,x'_{\ell-1},y'_{\ell-1}$ is $O_{K,\ell}((Kd)^{\ell-3}\cdot C_0d^{1/2}\cdot 8\ell)$, which is indeed of the form $O_{K,\ell}(d^{\ell-5/2})$.
\end{proof}

Combining Lemmas \ref{lem:internal vertex}, \ref{lem:two degeneracies} and \ref{lem:one degeneracy}, it follows that for any $\ell\geq 4$, if $G$ is a $C_{2\ell}^{\square}$-free $n$-vertex graph with average degree $d$ and maximum degree at most $Kd$, then the total weight of nice homomorphic copies of $C_{2\ell}^{\square}$ is $O_{K,\ell}(nd^{2\ell-2})$. This, combined with Lemma \ref{lem:lower bound nice}, implies the following result.

\begin{lemma}
    Let $\ell\geq 4$. If $C$ is sufficiently large compared to $K$ and $\ell$, then any clean, bipartite, $n$-vertex graph with average degree $d\geq Cn^{1/2}$ and maximum degree at most $Kd$ contains $C_{2\ell}^{\square}$ as a subgraph.
\end{lemma}

Using Lemmas~\ref{lem:JiangSeiverK-almost} and~\ref{lem:cleaned subgraph}, this proves Theorem~\ref{thm:Main1}.

\section{Concluding remarks}\label{sec:conclude}
In this paper, we gave optimal bounds for the order of magnitude of the extremal number for several classes of bipartite graphs arising from geometric shapes, including prisms, grids, hexagonal tilings and certain quadrangulations of the cylinder and the torus. Our method can be used to give tight bounds for several other families of graphs as well, some of which we shall briefly discuss here.

Let $R_{s,t}$ be the $s\times t$ hexagonal grid, the graph obtained from replacing each copy of unit $C_{4}$ in the $s\times t$ normal grid with a copy of $C_{6}$ in the way depicted in Figure~\ref{fig:RiEmbedding}. Lemma~\ref{lem:findgoodpaths} implies the following tight upper bound for $\textup{ex}(n,R_{s,t})$.

\begin{theorem}\label{thm:Rtt}
    For any integers $s,t\ge 2$, we have
\begin{equation*}
    \textup{ex}(n,R_{s,t})=\Theta(n^{4/3}).
\end{equation*}

\end{theorem}

\begin{figure}[t]
   \centering
   \includegraphics[width=12cm]{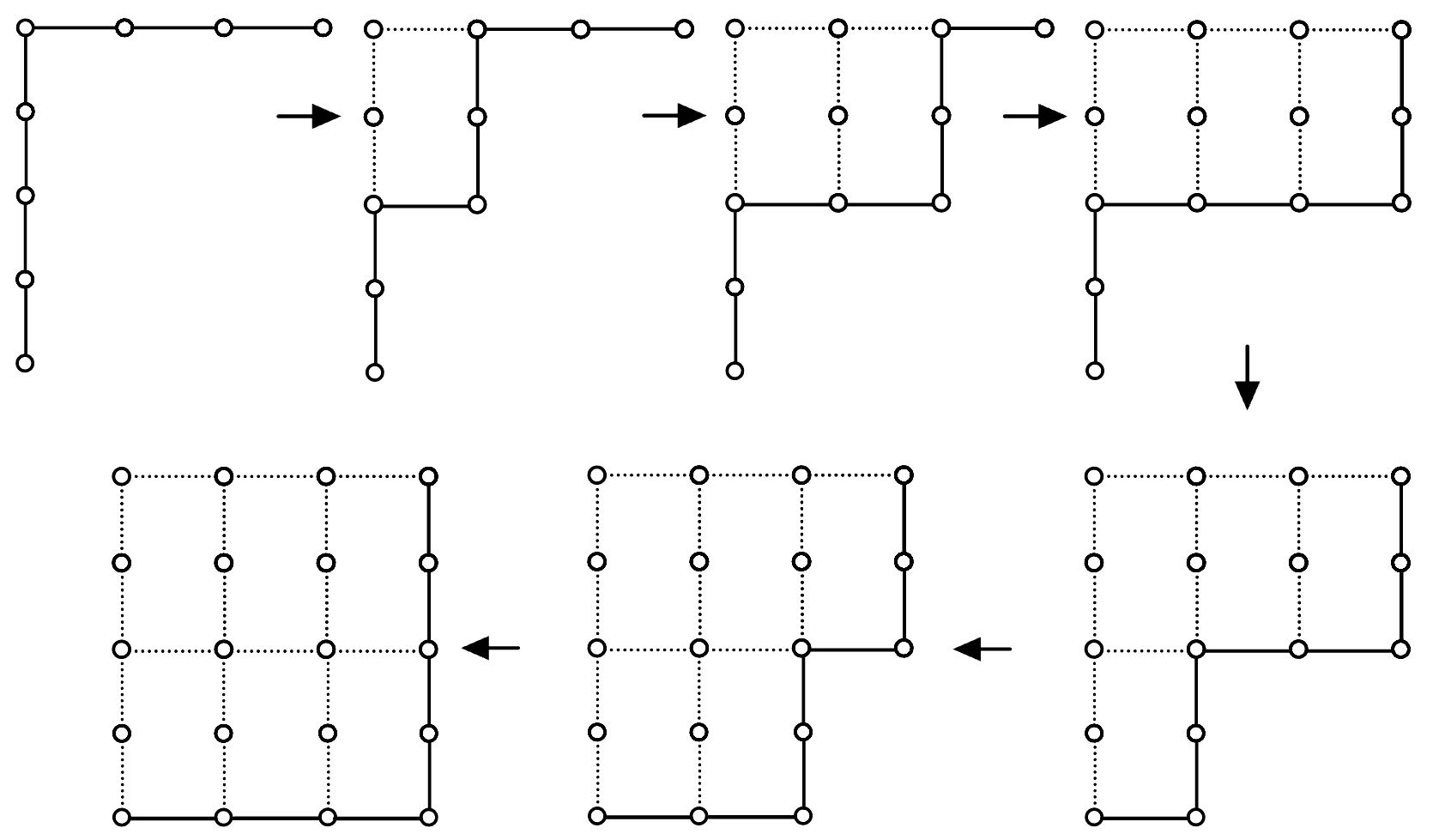}
   \caption{The process of embedding $R_{4,3}$}
   \label{fig:RiEmbedding}
\end{figure}

As a corollary, we obtain a tight upper bound for the Tur\'{a}n number of the rooted $2$-blowup of a certain balanced rooted tree. Let $T_{s,2s-1}$ be the tree obtained from an $s$-vertex path $x_{1}x_{2}\cdots x_{s}$ by adding $s$ new vertices $y_{1},y_{2}\ldots y_{s}$ and $s$ new edges $x_{1}y_{1},x_{2}y_{2}\ldots,x_{s}y_{s}$. The rooted $\ell$-blowup of $T_{s,2s-1}$, denoted $T_{s,2s-1}^{\ell}$, is obtained by taking $\ell$ vertex-disjoint copies of $T_{s,2s-1}$ and identifying the $\ell$ copies of $y_i$ for each $1\le i\le s$. Balanced rooted trees played a key role in the resolution of the rational exponents conjecture for families~\cite{2018JEMSBukh}. Later Jiang, Ma and Yepremyan~\cite{2021CPCJiang}, and Kang, Kim and Liu~\cite{2021JCTBKang} showed that $\textup{ex}(n,T_{3,5}^{\ell})=O(n^{7/5})$ and $\textup{ex}(n,T_{4,7}^{\ell})=O(n^{10/7})$, respectively, both of which are tight via the random algebraic method~\cite{2018JEMSBukh} when $\ell$ is very large. As $T_{s,2s-1}^{2}=R_{s,2}$, we have the following tight result for $T_{s,2s-1}^{2}$.

\begin{cor}\label{thm:2Blowup}
    For any integer $s\ge 2$, we have
    \begin{equation*}
        \textup{ex}(n,T_{s,2s-1}^{2})=\Theta(n^{4/3}).
    \end{equation*}
\end{cor}
Since $\ex(n,T_{s,2s-1}^{\ell})=\Omega(n^{1+\frac{s-1}{2s-1}})$ for sufficiently large $\ell$ \cite{2018JEMSBukh}, Corollary~\ref{thm:2Blowup} shows that the extremal number of $T_{s,2s-1}^{\ell}$ exhibits a phase transition phenomenon in relation to the number of blown-up copies.

Furthermore, we remark that Lemma~\ref{lem:findgoodpaths} easily implies that $\textup{ex}(n,\theta_{3,\ell})=\Theta(n^{4/3})$ for any $\ell\ge 2$, a result due to Faudree and Simonovits~\cite{1983ThetaGraph}, where $\theta_{3,\ell}$ is the graph consisting of $\ell$ internally vertex-disjoint paths of length 3 with the same endpoints.

Finally, as pointed out to us by Zach Hunter (private communication), the method applied in the proof of Theorem~\ref{thm:normal grid} can also be used to show that the extremal number of the rhombille tiling (see Figure~\ref{fig:Rhombille tiling}) is of order $n^{3/2}$. 

Our results also motivate some new directions to explore.

\begin{figure}[h]
\centering
\begin{minipage}[t]{0.56\textwidth}
\centering
\includegraphics[width=6cm]{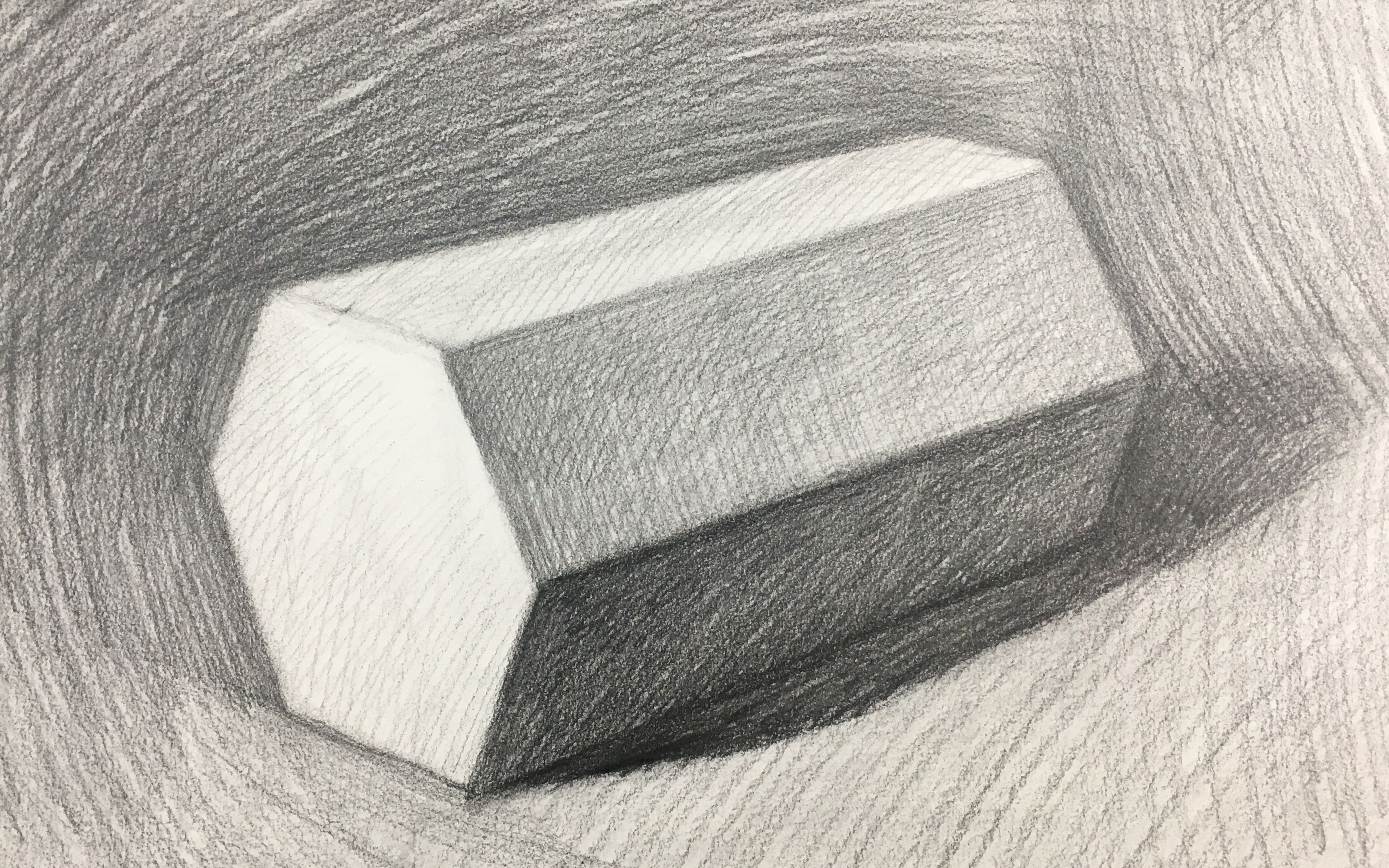}
\caption{$C_{6}^{\square}$}
\label{Hexagonal Tiling}
\end{minipage}
\hspace{-14mm}
\begin{minipage}[t]{0.4\textwidth}
\centering
\includegraphics[width=5cm]{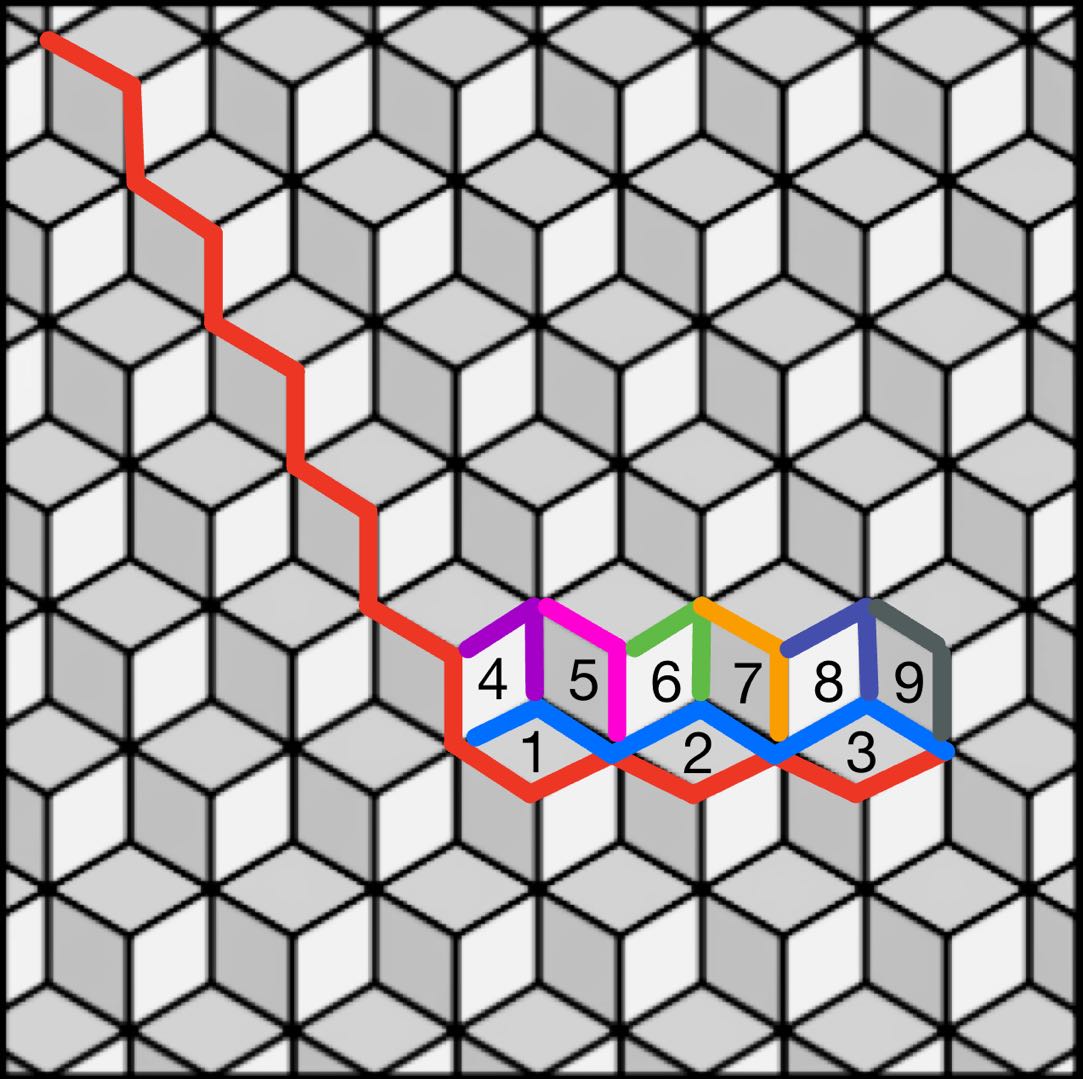}
\caption{Rhombille tiling}
\label{fig:Rhombille tiling}
\end{minipage}
\end{figure}

\begin{itemize}
    \item \textbf{Hexagonal prism $C_{6}^{\square}$}. An open problem left in this paper is determining the extremal number of $C_{6}^{\square}$. We put forward the following conjecture.
\begin{conj}
    $\ex(n,C_{6}^{\square})=\Theta(n^{3/2})$.
\end{conj}

    \item \textbf{The grid}.
    We proved that (for large $n$) the extremal number of the $t\times t$ grid satisfies $\ex(n,F_{t,t})\leq 5t^{3/2}n^{3/2}$. On the other hand, the lower bound of Brada\v{c}, Janzer, Sudakov and Tomon \cite{2022TuranGrid} is $\ex(n,F_{t,t})= \Omega(t^{1/2}n^{3/2})$. It would be interesting to determine the correct dependence on $t$.

\item \textbf{Tilings of hyperbolic plane}.
While the Tur\'an problem for both the grid and the hexagonal tiling in the plane are resolved (up to multiplicative constants), it is natural to study regular tilings of hyperbolic planes built up by 4-cycles and 6-cycles. In particular, we would like to propose the order-5 square tiling and order-6 hexagonal tiling, see Figures~\ref{fig:hyperbolic} and~\ref{fig:order 6}.

\begin{figure}[h]
\centering
\begin{minipage}[t]{0.4\textwidth}
\centering
\includegraphics[width=5cm]{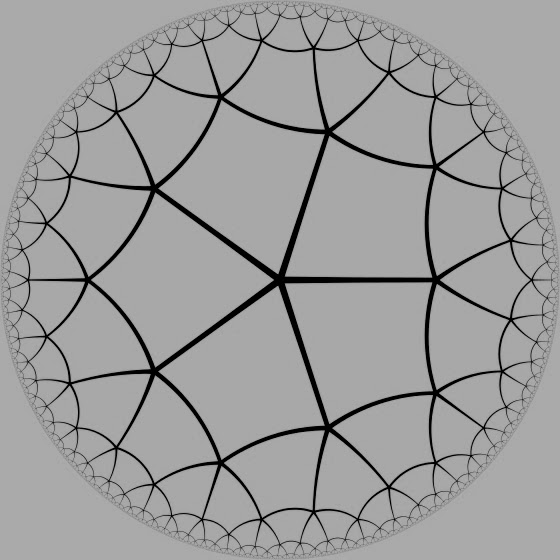}
\caption{Order-5 square tiling}
\label{fig:hyperbolic}
\end{minipage}
\begin{minipage}[t]{0.4\textwidth}
\centering
\includegraphics[width=5cm]{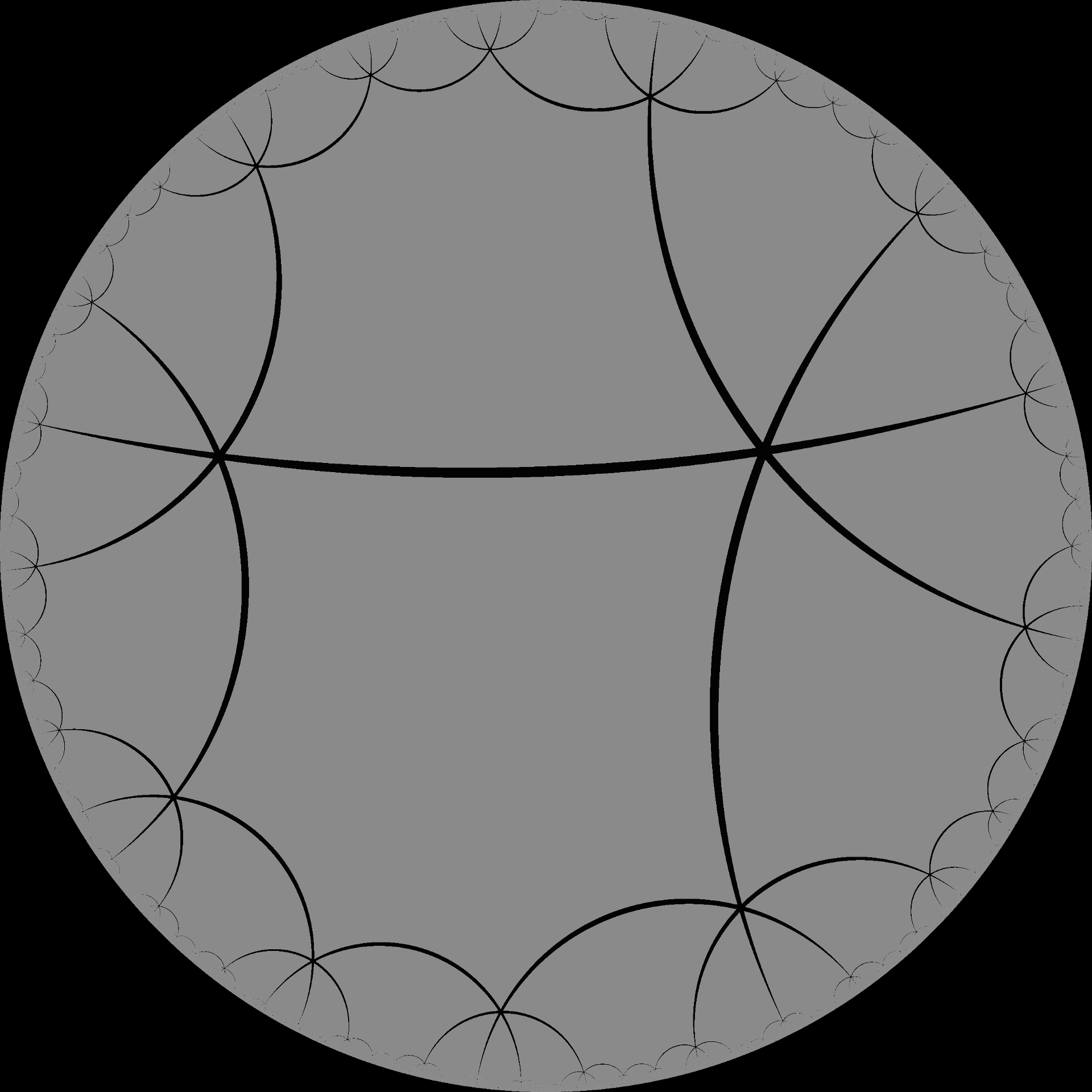}
\caption{Order-6 hexagonal tiling}
\label{fig:order 6}
\end{minipage}
\end{figure}

\item \textbf{Graphs with girth 4 or 6}. In the history of the Tur\'{a}n problem, there have been rich developments in the study of the graphs related to $4$-cycles or $6$-cycles. A natural problem is to find out when a bipartite graph $H$ containing $C_4$ (resp. $C_6$) has extremal number $\ex(n,H)=\Theta(n^{3/2})$ (resp. $\Theta(n^{4/3})$).
    
\end{itemize}

   \section*{Acknowledgement}
  We would like to express our gratitude to the anonymous reviewer for the detailed and constructive comments which are helpful to the improvement of the technical presentation of this paper.

\bibliographystyle{abbrv}
\bibliography{TightTuran}

\end{document}